\documentclass[11pt,a4paper]{article}
\usepackage[T1]{fontenc}
\usepackage[french,english]{babel}
\usepackage{amsmath}
\usepackage{caption}
\usepackage{amsfonts}
\usepackage{amssymb}
\usepackage{graphicx}
\usepackage{mathtools}
\usepackage{enumitem}
\usepackage[dvipsnames]{xcolor}
\usepackage[left=2.5cm,right=2.5cm,top=2.3cm,bottom=2.3cm]{geometry}
\usepackage{amsthm}
\usepackage{subcaption}
\usepackage{float}
\usepackage{authblk}
\usepackage{hyperref} 
\usepackage{bm}
\hypersetup{colorlinks = true, linkcolor=red, citecolor=blue}

\newcommand{\psh}[2]{\left\langle #1 \, , #2 \right\rangle \xspace}

\newcommand{\sumlim}[2]{\sum\limits_{#1}^{#2}\xspace}
\newcommand{\itg}[4]{\int_{#1}^{#2} #3 \, \mathrm{d}#4 \xspace}

\newcommand{\R}{\mathbb R}
\newcommand{\N}{\mathbb N}
\newcommand{\Z}{\mathbb Z}

\newcommand{\cO}{{\mathcal O}}

\newcommand{\cE}{{\mathcal E}}

\begin{document}
\numberwithin{equation}{section}
\newtheorem{theo}{Theorem}[section]
\newtheorem{prop}[theo]{Proposition}
\newtheorem{note}[theo]{Remark}
\newtheorem{lem}[theo]{Lemma}
\newtheorem{cor}[theo]{Corollary}
\newtheorem{definition}[theo]{Definition}
\newtheorem{assumption}{Assumption}

\title{Projector augmented-wave method: an analysis in a one-dimensional setting}
\author[1]{M.-S. Dupuy}
\affil[1]{Univ. Paris Diderot, Sorbonne Paris Cit\'e, Laboratoire Jacques-Louis Lions, UMR 7598, UPMC, CNRS, F-75205 Paris, France}

\renewcommand\Affilfont{\itshape\small}

\maketitle
\begin{abstract}
\begin{small}
In this article, a numerical analysis of the projector augmented-wave (PAW) method is presented, restricted to the case of dimension one with Dirac potentials modeling the nuclei in a periodic setting. The PAW method is widely used in electronic ab initio calculations, in conjunction with  pseudopotentials. It consists in replacing the original electronic Hamiltonian $H$ by a pseudo-Hamiltonian $H^{PAW}$ via the PAW transformation acting in balls around each nuclei. Formally, the new eigenvalue problem has the same eigenvalues as $H$ and smoother eigenfunctions. In practice, the pseudo-Hamiltonian $H^{PAW}$ has to be truncated, introducing an error that is rarely analyzed. In this paper, error estimates on the lowest PAW eigenvalue are proved for the one-dimensional periodic Schr\"odinger operator with double Dirac potentials.
\end{small}
\end{abstract}

\section*{Introduction}

In solid-state physics, to take advantage of the periodicity of the system, plane-wave methods are often the method of choice. However, Coulomb potentials located at each nucleus give rise to cusps on the eigenfunctions that impede the convergence rate of plane-wave expansions. Moreover, orthogonality to the core states implies fast oscillations of the valence state eigenfunctions that are difficult to approximate with plane-wave basis of moderate size. The PAW method \cite{blochl94} addresses both issues and has become a very popular tool over the years. It has been successfully implemented in different electronic structure simulation codes (ABINIT \cite{torrent2008337}, VASP \cite{kresse99}) and has been adapted to the computations of various chemical properties \cite{audouze2006projector,pickard2001all}. It relies on an invertible transformation acting locally around each nucleus, mapping solutions of an atomic wave function to a smoother and slowly varying function. Moreover, because of the particular form of the PAW transformation, it is possible to use pseudopotentials \cite{kleinman1982efficacious,troullier1991efficient} in a consistent way. Hence, the PAW eigenfunctions are smoother and because of the invertibility of the PAW transformation, the sought eigenvalues are the same. However, the theoretical PAW equations involve infinite expansions which have to be truncated in practice. Doing so, the PAW method introduces an error that is rarely analyzed. 

In this paper, the PAW method is applied to the one-dimensional double Dirac potential Hamiltonian whose eigenfunctions display a cusp at the location of the Dirac potentials that is reminiscent of the Kato cusp condition \cite{kato1957eigenfunctions}. Error estimates on the lowest PAW eigenvalue are proved for several choices of PAW parameters. The present analysis relies on some results on the variational PAW method (VPAW method) \cite{blanc2017, blanc2017vpaw1d} which is a slight modification of the original PAW method. Contrary to the PAW method, the VPAW generalized eigenvalue problem is in one-to-one correspondence with the original eigenvalue problem. By estimating the difference between the PAW and VPAW generalized eigenvalue problems, error estimates on the lowest PAW generalized eigenvalue are found.

\section{The PAW method in a one-dimensional setting}
\label{sec:setting}

A general overview of the VPAW and PAW methods for 3-D electronic Hamiltonians may be found in \cite{blanc2017} for the molecular setting and in \cite{blanc2017vpaw1d} for crystals. Here, the presentation of the VPAW and PAW methods is limited to the application to the 1-D periodic Schr\"odinger operator with double Dirac potentials. 

\subsection{The double Dirac potential Schr\"odinger operator}

We are interested in the lowest eigenvalue	 of the 1-D periodic Schr\"odinger operator $H$  on $L^2_{\mathrm{per}}(0,1) := \{ f \in L^2_{\mathrm{loc}} (\mathbb{R}) \ | \ f \text{ 1-periodic} \}$ with form domain $H^1_{\mathrm{per}}(0,1) := \{ f \in H^1_{\mathrm{loc}} (\mathbb{R}) \ | \ f \text{ 1-periodic}\}$:
\begin{equation}
H = -\frac{\mathrm{d}^2}{\mathrm{d}x^2} - Z_0 \sumlim{k \in \Z}{} \delta_k - Z_a \sumlim{k \in \Z}{} \delta_{k+a},
\label{eq:H_mol}
\end{equation}
where $0 < a < 1$, $Z_0, Z_a >0$. 

A mathematical analysis has been carried out in \cite{cances2017discretization}. There are two negative eigenvalues $E_0 = -\omega_0^2$ and $E_1 = -\omega_1^2$ which are given by the zeros of the function 
$$
f(\omega) = 2 \omega^2 (1- \cosh (\omega)) + (Z_0 + Z_a) \omega \sinh (\omega) - Z_0 Z_a \sinh(a \omega) \sinh((1-a)\omega).
$$
The corresponding eigenfunctions are 
\begin{equation*}
\psi_k(x) = 
\begin{cases}
A_{1,k} \cosh (\omega_k x) + B_{1,k} \sinh(\omega_k x) \ , \ 0 \leq x \leq a, \\
A_{2,k} \cosh (\omega_k x) + B_{2,k} \sinh(\omega_k x) \ , \ a \leq x \leq 1, \\
\end{cases}
\end{equation*}
where the coefficients $A_{1,k}$, $A_{2,k}$, $B_{1,k}$ and $B_{2,k}$ are determined by the continuity conditions and the derivative jumps at $0$ and $a$.

There is an infinity of positive eigenvalues $E_{k+2} = \omega_{k+2}^2$ which are given by the $k$-th zero of the function :
\begin{equation*}
f(\omega) = 2 \omega^2 (1- \cos (\omega)) + (Z_0 + Z_a) \omega \sin (\omega) + Z_0 Z_a \sin(a \omega) \sin((1-a)\omega),
\end{equation*}
and the corresponding eigenfunctions $H \psi_k = \omega_k^2 \psi_k$ are   
\begin{equation}
\psi_k(x) = 
\begin{cases}
A_{1,k} \cos ( \omega_k x) + B_{1,k} \sin( \omega_k x) \ , \ 0 \leq x \leq a, \\
A_{2,k} \cos ( \omega_k x) + B_{2,k} \sin( \omega_k x) \ , \ a \leq x \leq 1, \\
\end{cases}
\end{equation}
where again the coefficients $A_{1,k}$, $A_{2,k}$, $B_{1,k}$ and $B_{2,k}$ are determined by the continuity conditions and the derivative jumps at $0$ and $a$.
Notice that the eigenfunctions of $H$ have a first derivative jump that is similar to the Kato cusp condition satisfied by the solutions of 3D electronic Hamiltonian \cite{kato1957eigenfunctions}:
$$
\psi_k'(0_+) - \psi_k'(0_-) = -Z_0 \psi_k(0). 
$$

\subsection{The PAW method}
\label{sec:HPAW-pseudo}

\subsubsection{General principle}

The PAW method consists in replacing the original eigenvalue problem $H \psi = E\psi$ by the generalized eigenvalue problem
\begin{equation}
\label{eq:PAW_theorie}
(\mathrm{Id} + T^*) H (\mathrm{Id} + T)\widetilde{\psi} = E (\mathrm{Id} + T^*) (\mathrm{Id} + T)\widetilde{\psi},
\end{equation}
where $\mathrm{Id}+T$ is an invertible operator. It is clear that \eqref{eq:PAW_theorie} is equivalent to $H \psi = E \psi$ where $\psi = (\mathrm{Id}+T) \widetilde{\psi}$. 

The transformation $T$ is the sum of two operators acting in regions near the atomic sites that do not overlap (\emph{i.e.} $T_0 T_a = T_a T_0 = 0$)
$$
T = T_0 + T_a, \quad T_0 = \sumlim{i=0}{\infty} (\phi_i^0 - \widetilde{\phi}_i^0)\psh{\widetilde{p}_i^0}{\bm{\cdot}}, \quad T_a = \sumlim{i=0}{\infty} (\phi_i^a - \widetilde{\phi}_i^a)\psh{\widetilde{p}_i^a}{\bm{\cdot}},
$$
where $\psh{\bm{\cdot}}{\bm{\cdot}}$ denotes the $L^2_\mathrm{per}(0,1)$ scalar product. 

The atomic wave functions $(\phi_j^0)_{j \in \N}$ are solutions of an atomic eigenvalue problem
$$
H_0 \phi_j^0 := -\frac{\mathrm{d}^2\phi_j^0}{\mathrm{d}x^2} - Z_0 \sumlim{k \in \Z}{} \delta_k\phi_j^0   = \epsilon_j^0 \phi_j^0,
$$
and the pseudo wave functions $(\widetilde{\phi}_j^0)_{j \in \N}$ and the projector functions $(\widetilde{p}_j^0)_{j\in \N}$ satisfy the following conditions :
\begin{enumerate}
\item for each $j \in \N$, 
\begin{itemize}
\item for $x \in \R \setminus \bigcup\limits_{k \in \Z}[-\eta+k,\eta+k]$, $\widetilde{\phi}_j^0(x) = \phi_j^0(x)$;
\item $\widetilde{\phi}_j^0$ restricted to $\bigcup\limits_{k \in \Z}[-\eta+k,\eta+k]$ is a smooth function;
\end{itemize}

\item for each $j \in \N$, $\mathrm{supp} \ \widetilde{p}_j^0 \subset \bigcup\limits_{k \in \Z} [-\eta+k,\eta+k]$;
\item the families $(\widetilde{\phi}_j^0 |_{[-\eta,\eta]})_{j \in \N}$ and $(\widetilde{p}_j^0|_{[-\eta,\eta]})_{j\in \N}$ form a Riesz basis of $L^2(-\eta,\eta)$, \emph{i.e.} 
$$
\forall \, j,k \in \N, \ \itg{-\eta}{\eta}{\widetilde{p}_k^0(x)\widetilde{\phi}_j^0(x)}{x} = \delta_{kj},
$$
and for any $f \in L^2(-\eta,\eta)$, we have
\begin{equation}
\label{eq:PAW_completeness}
\sumlim{k=0}{\infty} \psh{\widetilde{p}_k^0}{f} \widetilde{\phi}_k^0(x) = f(x), \quad \text{for a.a. } x \in \bigcup\limits_{k \in \Z} [-\eta + k, \eta + k].
\end{equation}
\end{enumerate} 
Similarly,  $(\phi^a_i)_{i \in \N^*}$ are eigenfunctions of the operator $H_a = -\frac{\mathrm{d}^2}{\mathrm{d}x^2} - Z_0 \sumlim{k \in \Z}{} \delta_{a+k}  $, the pseudo wave functions $(\widetilde{\phi}^a_j)_{j \in \N^*}$ and the projector functions $(\widetilde{p}^a_j)_{j\in \N^*}$ are defined as above. 
\newline
 
The relation \eqref{eq:PAW_completeness} enables one to write the expression of $(\mathrm{Id}+T^*)H(\mathrm{Id}+T)$ and \linebreak $(\mathrm{Id}+T^*)(\mathrm{Id}+T)$ as
\begin{equation}
\label{eq:H_fullPAW}
(\mathrm{Id}+T^*) H (\mathrm{Id}+T) = H +  \sumlim{\substack{ i,j=0 \\ I = \lbrace 0, a \rbrace }}{\infty} \widetilde{p}^I_i 
\left( \psh{\phi_i^I}{H \phi_j^I} - \psh{\widetilde{\phi}^I_i}{H \widetilde{\phi}^I_j} \right) \psh{\widetilde{p}^I_j}{\bm{\cdot}} ,
\end{equation} 
and
\begin{equation}
\label{eq:S_fullPAW}
(\mathrm{Id}+T^*) (\mathrm{Id}+T) = \mathrm{Id} +  \sumlim{\substack{ i,j=0 \\ I = \lbrace 0, a \rbrace }}{\infty} \widetilde{p}^I_i
\left( \psh{\phi_i^I}{ \phi_j^I} - \psh{\widetilde{\phi}^I_i}{ \widetilde{\phi}^I_j} \right) \psh{\widetilde{p}^I_j}{\bm{\cdot}} .
\end{equation}

\subsubsection{Introduction of a pseudopotential}

A further modification is possible. As the pseudo wave functions $\widetilde{\phi}_i^0$ (resp. $\widetilde{\phi}_i^a$) are equal to $\phi_i^0$ (resp. $\phi_i^a$) outside $\bigcup\limits_{k \in \Z} [-\eta + k, \eta + k]$ (resp. $\bigcup\limits_{k \in \Z} [a-\eta + k, a+\eta + k]$), the integrals appearing in \eqref{eq:H_fullPAW} can be truncated to the interval $(-\eta,\eta)$ (resp. $(a-\eta,a+\eta)$). Doing so, another expression of $(\mathrm{Id}+T^*) H (\mathrm{Id}+T)$ can be obtained :
\begin{equation*}
(\mathrm{Id}+T^*) H (\mathrm{Id}+T) = H +  \sumlim{\substack{ i,j=0 \\ I = \lbrace 0, a \rbrace }}{\infty} \tilde{p}^I_i
\left( \psh{\phi_i^I}{H \phi_j^I}_{I,\eta} - \psh{\widetilde{\phi}^I_i}{H \widetilde{\phi}^I_j}_{I,\eta} \right) \psh{\tilde{p}^I_j}{\bm{\cdot}} ,
\end{equation*}
where 
\begin{equation*}
\psh{f}{g}_{I,\eta} =  
\begin{cases}
\itg{-\eta}{\eta}{f(x)g(x)}{x}, \quad & \text{when } I=0, \\
\itg{a-\eta}{a+\eta}{f(x)g(x)}{x}, \quad &  \text{when } I=a. 
\end{cases}
\end{equation*}
Using this expression of the operator $H^{PAW}$, it is possible to introduce a smooth 1-periodic potential $\chi_\epsilon = \sumlim{k \in \Z}{} \frac{1}{\epsilon} \chi \left( \frac{\cdot -k}{\epsilon} \right)$ with $\epsilon \leq \eta$, such that 
\begin{enumerate}
\item $\chi$ is a smooth nonnegative function with support $[-1,1]$ and $\itg{-1}{1}{\chi(x)}{x} = 1$;
\item $\chi_\epsilon \underset{\epsilon \to 0}{\longrightarrow} \sumlim{k \in \Z}{} \delta_k$ in $H^{-1}_\mathrm{per}(0,1)$.
\end{enumerate}
The potential $\chi_\epsilon$ will be called a \emph{pseudopotential} in the following. 
\\

The expression of $(\mathrm{Id}+T^*) H (\mathrm{Id}+T)$ becomes
\begin{equation}\label{eq:1}
(\mathrm{Id}+T^*) H (\mathrm{Id}+T) = H_\mathrm{ps} + \sumlim{\substack{ i,j=0 \\ I = \lbrace 0, a \rbrace }}{\infty} \widetilde{p}^I_i 
\left( \psh{\phi_i^I}{H \phi_j^I}_{I,\eta} - \psh{\widetilde{\phi}^I_i}{H_\mathrm{ps} \widetilde{\phi}^I_j}_{I,\eta} \right) \psh{\widetilde{p}^I_j}{\bm{\cdot}} ,
\end{equation}
with 
\begin{displaymath}
H_\mathrm{ps} =  \frac{\mathrm{d}^2}{\mathrm{d}x^2} - Z_0 \chi_\epsilon - Z_a \chi_\epsilon(\cdot -a).
\end{displaymath}

\subsection{The PAW method in practice}

In practice, the double sums appearing in the operators \eqref{eq:H_fullPAW}, \eqref{eq:S_fullPAW} and \eqref{eq:1} have to be truncated to some level $N$. Doing so, the identity $\psi = (\mathrm{Id}+T) \widetilde{\psi}$ is lost and the eigenvalues of the truncated equations are not equal to those of the original operator $H$ \eqref{eq:H_mol}. The PAW method introduces an error that will be estimated in the rest of paper. First, we define the PAW functions appearing in \eqref{eq:H_fullPAW}, \eqref{eq:S_fullPAW} and \eqref{eq:1}.

\subsubsection{Generation of the PAW functions}
\label{subsec:generation}

For the double Dirac potential Hamiltonian, the PAW functions are defined as follows.
\paragraph{Atomic wave functions $\phi_k^0$}
As mentioned earlier, the atomic wave functions $(\phi_k^0)_{1 \leq k \leq N}$ are eigenfunctions of the Hamiltonian $H_0$
$$
H_0 = -\frac{\mathrm{d}^2}{\mathrm{d}x^2} - Z_0 \sumlim{k \in \Z}{} \delta_k .
$$
By parity, each eigenfunction of this operator is either even or odd. The odd eigenfunctions are in fact $x \mapsto \sin (2 \pi k x)$ and the even ones are the $1$-periodic functions such that
$$
\begin{cases}
\phi_0^0(x) := \cosh(\omega_0 (x - \tfrac{1}{2})) & \text{for } x \in [0,1], \\
\phi_k^0(x) := \cos(\omega_k (x-\tfrac{1}{2}) ) & \text{for } x \in [0,1], \ k \in \N^*,
\end{cases}
$$
In the sequel (and in particular in \eqref{eq:H_N} and \eqref{eq:H_PAW_pseudo} below), only the non-smooth thus even eigenfunctions $(\phi_i^0)_{1 \leq i \leq N}$ are selected. The corresponding eigenvalues are denoted by $(\epsilon_i^0)_{1 \leq i \leq N}$:
$$
H_0 \phi_i^0 = \epsilon_i^0 \phi_i^0.
$$

\paragraph{Pseudo wave function $\widetilde{\phi}^0_i$}
The pseudo wave functions $(\widetilde{\phi}_i^0)_{1 \leq i \leq N} \in \left(H^1_\mathrm{per}(0,1)\right)^N$ are defined as follows:
\begin{enumerate}
\item for $x \notin \bigcup\limits_{k \in \Z} [-\eta + k, \eta+k]$, $\widetilde{\phi}_i^0(x) = \phi_i^0(x)$.
\item for $x \in \bigcup\limits_{k \in \Z} [-\eta + k, \eta+k]$, $\widetilde{\phi}^0_i$ is an even polynomial of degree at most $2d-2$, $d \geq N$.  
\item $\widetilde{\phi}^0_i$ is $C^{d-1}$ at $\eta$ \emph{i.e.} $(\widetilde{\phi}_i^0)^{(k)}(\eta) = (\phi_i^0)^{(k)}(\eta)$ for $0 \leq k \leq d-1$.
\end{enumerate}

\paragraph{Projector functions $\widetilde{p}_i^0$}
Let $\rho$ be a positive, smooth function with support included in $[-1,1]$ and  $\rho_\eta(t) = \sumlim{k\in \Z}{} \rho\left(\tfrac{t-k}{\eta}\right)$. The projector functions $(\widetilde{p}_i^0)_{1 \leq i \leq N}$ are obtained by an orthogonalization procedure from the functions $p_i^0(t) = \rho_\eta (t) \widetilde{\phi}_i^0(t)$ in order to satisfy the duality condition :
$$
\psh{\widetilde{p}_i^0}{\widetilde{\phi}_j^0} = \delta_{ij} .
$$
More precisely, the matrix $B_{ij} := \psh{p_i^0}{\widetilde{\phi}_j^0}$ is computed and inverted to obtain the projector functions
$$
\widetilde{p}_k^0 = \sumlim{j=1}{N} (B^{-1})_{kj} p_j^0. 
$$
The matrix $B$ is the Gram matrix of the functions $(\widetilde{\phi}_j^0)_{1\leq j \leq N}$ for the weight $\rho_\eta$. The orthogonalization is possible only if the family $(\widetilde{\phi}_i^0)_{1 \leq i \leq N}$ is linearly independent - thus necessarily $d \geq N$.

\subsubsection{The eigenvalue problems}

For the case \emph{without} pseudopotentials, the PAW eigenvalue problem is given by
\begin{equation}
\label{eq:PAW_eig_pb}
H^N f = E^{(\eta)} S^N f,
\end{equation}
where
\begin{equation}
\label{eq:H_N}
H^N = H + \sumlim{\substack{ i,j=1 \\ I = \lbrace 0, a \rbrace }}{N} \widetilde{p}^I_i 
\left( \psh{\phi_i^I}{H \phi_j^I} - \psh{\widetilde{\phi}^I_i}{H \widetilde{\phi}^I_j} \right) \psh{\widetilde{p}^I_j}{\bm{\cdot}} ,
\end{equation} 
and
\begin{equation}
\label{eq:S_N}
S^N = \mathrm{Id} + \sumlim{\substack{ i,j=1 \\ I = \lbrace 0, a \rbrace }}{N} \widetilde{p}^I_i
\left( \psh{\phi_i^I}{ \phi_j^I} - \psh{\widetilde{\phi}^I_i}{ \widetilde{\phi}^I_j} \right) \psh{\widetilde{p}^I_j}{\bm{\cdot}} .
\end{equation}
The practical interest in solving the eigenvalue problem \eqref{eq:PAW_eig_pb} is very limited since this version of the PAW method does not remove the singularity caused by the Dirac potentials. The next eigenvalue problem where the Dirac potentials are replaced by smoother potentials is closer to the implementation of the PAW method in practice.
\newline

For the case \emph{with} pseudopotentials, the PAW eigenvalue problem becomes
\begin{equation}
\label{eq:PAW_eig_pb_pseudo}
H^{PAW} f = E^{PAW} S^{PAW} f,
\end{equation}
where
\begin{equation}
\label{eq:H_PAW_pseudo}
H^{PAW} = H_\mathrm{ps} + \sumlim{\substack{ i,j=1 \\ I = \lbrace 0, a \rbrace }}{N} \widetilde{p}_i^I \left( \psh{\phi_i^I}{H \phi_j^I} - \psh{\widetilde{\phi}_i^I}{H_\mathrm{ps} \widetilde{\phi}_j^I} \right) \psh{\widetilde{p}_j^I}{\bm{\cdot}} ,
\end{equation} 
and
\begin{equation}
\label{eq:S_PAW_pseudo}
S^{PAW} = S^N = \mathrm{Id} + \sumlim{\substack{ i,j=1 \\ I = \lbrace 0, a \rbrace }}{N} \widetilde{p}_i^I \left( \psh{\phi_i^I}{\phi_j^I} - \psh{\widetilde{\phi}_i^I}{\widetilde{\phi}_j^I} \right) \psh{\widetilde{p}_j^I}{\bm{\cdot}} .
\end{equation}

If the projector functions $(\widetilde{p}_i)_{1 \leq i \leq N}$ are smooth, then the eigenfunctions $f$ in \eqref{eq:PAW_eig_pb_pseudo} are smooth as well, and their plane-wave expansions converge very quickly. Thus, if the difference $|E^{PAW}-E|$ is smaller than a desired accuracy, it is more interesting to solve \eqref{eq:PAW_eig_pb_pseudo} than the original eigenvalue problem. However, an estimate on the difference $|E^{PAW}-E|$ is needed in order to justify the use of the PAW method. To the best of our knowledge, there exists no estimation of this error except a heuristic analysis in the seminal work of Bl\"ochl (\cite{blochl94}, Sections VII.B and VII.C). However, his analysis relies on an expansion of the eigenvalue in $f - \sumlim{i=1}{N} \psh{\widetilde{p}_i}{f} \widetilde{\phi}_i$ which goes to $0$ if the families $(\widetilde{p}_i)_{i \in \N^*}$ and $(\widetilde{\phi}_i)_{i \in \N^*}$ form a Riesz basis, but a convergence rate of the expansion in the Riesz basis is not given. Moreover the inclusion of a pseudopotential in the PAW treatment is not taken into account. 
\newline

The goal of this paper is to provide error estimates on the lowest PAW eigenvalue of problems \eqref{eq:PAW_eig_pb} and \eqref{eq:PAW_eig_pb_pseudo}.  To prove this result, the PAW method is interpreted as a perturbation of the VPAW method introduced in \cite{blanc2017,blanc2017vpaw1d} which has the same eigenvalues as the original problem. In the following, when we refer to the PAW method, it will be to the truncated equations \eqref{eq:PAW_eig_pb} or \eqref{eq:PAW_eig_pb_pseudo}.

\subsection{The VPAW method}
\label{subsec:VPAW1D}

The analysis of the PAW method relies on the connexion between the VPAW and the PAW methods. A brief description of the VPAW method is given in this subsection.

Like the PAW method, the principle of the VPAW method consists in replacing the original eigenvalue problem
$$
H\psi = E\psi,
$$
by the generalized eigenvalue problem:
\begin{equation}
(\mathrm{Id}+T_N^*) H (\mathrm{Id}+T_N) \widetilde{\psi} = E (\mathrm{Id}+T_N) (\mathrm{Id}+T_N) \widetilde{\psi},
\label{eq:H_VPAW_eig_pb}
\end{equation}
where $\mathrm{Id}+T_N$ is an invertible operator. Thus both problems have the same eigenvalues and it is straightforward to recover the eigenfunctions of the former from the generalized eigenfunctions of the latter:
\begin{equation*}
\psi = (\mathrm{Id}+T_N) \widetilde{\psi}.
\end{equation*}

Again, $T_N$ is the sum of two operators acting near the atomic sites 
\begin{equation}
\label{eq:T_N}
T_N = T_{0,N} + T_{a,N}.
\end{equation}
To define $T_{0,N}$, we fix an integer $N$ and a radius $0 < \eta < \min(\frac{a}{2}, \frac{1-a}{2})$ so that $T_{0,N}$ and $T_{a,N}$ act on two disjoint regions $\bigcup\limits_{k \in \Z} [-\eta+k,\eta+k]$ and $\bigcup\limits_{k \in \Z} [a-\eta+k,a+\eta+k]$ respectively. 

The operators $T_{0,N}$ and $T_{a,N}$ are given by
\begin{equation}
\label{eq:T0}
T_{0,N} = \sumlim{i=1}{N} (\phi_{i}^0 - \widetilde{\phi}_i^0) \psh{\widetilde{p}_i^0}{\bm{\cdot}}, \qquad T_{a,N} = \sumlim{i=1}{N} (\phi_{i}^a - \widetilde{\phi}_i^a) \psh{\widetilde{p}_i^a}{\bm{\cdot}} \ ,
\end{equation}
with the same functions $\phi_i^I$, $\widetilde{\phi}_i^I$ and $\widetilde{p}^I_i$, $I=0,a$ as in Section \ref{sec:HPAW-pseudo}. The only difference with the PAW method is that the sums appearing in \eqref{eq:T0} are \emph{finite}, thereby avoiding a truncation error. 

In the following, the VPAW operators are denoted by
\begin{equation}
\label{eq:H_VPAW}
\widetilde{H} = (\mathrm{Id} + T_N^*) H (\mathrm{Id} + T_N),
\end{equation}
and 
\begin{equation}
\label{eq:S_VPAW}
\widetilde{S} = (\mathrm{Id} + T_N^*) (\mathrm{Id} + T_N),
\end{equation}
\newline

A full analysis of the VPAW method can be found in \cite{blanc2017vpaw1d}. In this paper, we proved that the cusps at $0$ and $a$ of the eigenfunctions $\widetilde{\psi}$ are reduced by a factor $\eta^{2N}$ but the $d$-th derivative jumps introduced by the pseudo wave functions $\widetilde{\phi}_k$ blow up as $\eta$ goes to 0 at the rate $\eta^{1-d}$. Using Fourier methods to solve \eqref{eq:H_VPAW_eig_pb}, we observe an acceleration of convergence that can be tuned by the VPAW parameters $\eta$ -the cut-off radius- $N$ -the number of PAW functions used at each site- and $d$ -the smoothness of the PAW pseudo wave functions. 

\section{Main results}
\label{sec:results}

The PAW method is well-posed if the projector functions $(\widetilde{p}_i^I)_{1 \leq i \leq N}$ are well-defined. This question has already been addressed in \cite{blanc2017vpaw1d} where it is shown that we simply need to take $\eta < \eta_0$ for some positive $\eta_0$. 

\begin{assumption}
\label{assump:1}
Let $\eta_0 > 0$ such that for all $0 < \eta < \eta_0$, the projector functions $(\widetilde{p}_i)_{1 \leq i \leq N}$ in Section  \ref{subsec:generation} are well-defined.
\end{assumption}

Moreover since the analysis of the PAW error requires the VPAW method to be well-posed, the matrix $\left(\psh{\widetilde{p}_j^I}{\phi_k^I}\right)_{1 \leq j,k \leq N}$ is assumed to be invertible for $0 < \eta \leq \eta_0$. 

\begin{assumption}
\label{assump:2}
For all $0 < \eta < \eta_0$, the matrix  $\left(\psh{\widetilde{p}_j^I}{\phi_k^I}\right)_{1 \leq j,k \leq N}$ is invertible. 
\end{assumption}

Under these assumptions, the following theorems are established. Proofs  are gathered in Section \ref{sec:proof}.

\subsection{PAW method without pseudopotentials}

\begin{theo}
\label{thm:bound_trunc_PAW}
Let $\phi_i^I$, $\widetilde{\phi}_i^I$ and $\widetilde{p}_i^I$, for $i=1,\dots,N$ and $I=0,a$ be the functions defined in Section \ref{subsec:generation}. Suppose $\eta_0 > 0$ satisfies Assumption \ref{assump:1} and Assumption \ref{assump:2}. Let $E^{(\eta)}$ be the lowest eigenvalue of the generalized eigenvalue problem \eqref{eq:PAW_eig_pb}. Let $E_0$ be the lowest eigenvalue of $H$ \eqref{eq:H_mol}. Then there exists a positive constant $C$ independent of $\eta$ such that for all $0 < \eta \leq \eta_0$
\begin{equation}
\label{eq:thm_PAW_sans_pseudo}
-C \eta \leq E^{(\eta)} - E_0 \leq C \eta^{2N}.
\end{equation}
\end{theo}

The constant $C$ appearing in \eqref{eq:thm_PAW_sans_pseudo} (and in the theorems that will follow) depends on the other PAW parameters $N$ and $d$ in a nontrivial way. The upper bound is proved by using the VPAW eigenfunction $\widetilde{\psi}$ associated to the lowest eigenvalue $E_0$ for which we have precise estimates of the difference between the operators $H^{PAW}$ and $\widetilde{H}$. As expected (and confirmed by numerical simulations in Section \ref{subsec:PAW_simu_sans_pseudo}) the PAW method without pseudopotentials is \emph{not} variational. Moreover as the Dirac delta potentials are not removed, Fourier methods applied to the eigenvalue problem \eqref{eq:PAW_eig_pb} converge slowly. 

\subsection{PAW method with pseudopotentials}

The following theorems are stated for $\epsilon = \eta$, \emph{i.e.} when the support of the pseudopotential is equal to the acting region of the PAW method. Indeed, in the proof of Theorem \ref{thm:E_PAW_bound}, it appears  that worse estimates are obtained when a pseudopotential $\chi_\epsilon$ with $\epsilon < \eta$ is used.  

\begin{theo}
\label{thm:E_PAW_bound}
Let $\phi_i^I$, $\widetilde{\phi}_i^I$ and $\widetilde{p}_i^I$, for $i=1,\dots,N$ and $I=0,a$ be the functions defined in Section \ref{subsec:generation}. Suppose $\eta_0 > 0$ satisfies Assumption \ref{assump:1} and Assumption \ref{assump:2}.
Let $E^{PAW}$ the lowest eigenvalue of the generalized eigenvalue problem \eqref{eq:PAW_eig_pb_pseudo}. Let $E_0$ be the lowest eigenvalue of $H$ \eqref{eq:H_mol}. Then there exists a positive constant $C$ independent of $\eta$ such that for all $0 < \eta \leq \eta_0$
\begin{equation}
- C \eta \leq E^{PAW} - E_0 \leq C \eta^2.
\end{equation}
\end{theo}

Introducing a pseudopotential in $H^{PAW}$ worsens the upper bound on the PAW eigenvalue. This is due to our construction of the PAW method in Section \ref{sec:HPAW-pseudo} where only even PAW functions are considered. Incorporating odd PAW functions in the PAW treatment, it is possible to improve the upper bound on the PAW eigenvalue and recover the bound in Theorem \ref{thm:bound_trunc_PAW} (see Section \ref{subsec:odd_PAW}). 

As the cut-off radius $\eta$ goes to $0$, the lowest eigenvalue of the truncated PAW equations is closer to the exact eigenvalue. This is also observed in different implementations of the PAW method and is in fact one of the main guidelines: a small cutoff radius yields more accurate results \cite{jollet2014generation,holzwarth2001projector}. 

\begin{theo}
\label{thm:E_PAW_M}
Let $\phi_i^I$, $\widetilde{\phi}_i^I$ and $\widetilde{p}_i^I$, for $i=1,\dots,N$ and $I=0,a$ be the functions defined in Section \ref{subsec:generation}. Suppose $\eta_0 > 0$ satisfies Assumption \ref{assump:1} and Assumption \ref{assump:2}.
Let $E_M^{PAW}$ be the lowest eigenvalue of the variational approximation of \eqref{eq:PAW_eig_pb_pseudo}, with $H^{PAW}$ given by  \eqref{eq:H_PAW_pseudo} in a basis of $M$ plane waves. Let $E_0$ be the lowest eigenvalue of $H$ \eqref{eq:H_mol}. There exists a positive constant $C$ independent of $\eta$ and $M$ such that for all $0 < \eta < \eta_0$ and for all $n \in \N^*$
$$
\left| E_M^{PAW} - E_0 \right| \leq C \left( \eta + \frac{\eta^2}{(\eta M)^n} \right).
$$
\end{theo}

According to Theorem \ref{thm:E_PAW_M}, if we want to compute $E_0$ up to a desired accuracy $\varepsilon$, then it suffices to choose the PAW cut-off radius $\eta$ equal to $\frac{1}{C \varepsilon}$ and solve the PAW eigenvalue problem \eqref{eq:PAW_eig_pb_pseudo} with $M \geq \frac{1}{\eta}$ plane-waves. 

\begin{note}
Using more PAW functions does not improve the bound on the computed eigenvalue. It is due to the poor lower bound in Theorems \ref{thm:E_PAW_bound} and \ref{thm:E_PAW_improved_bound}.
Should the PAW method with odd functions (Section \ref{subsec:odd_PAW}) be variational, we would know a priori that $E^{PAW} \geq E_0$. Therefore, we could prove the estimate
$$
0 < E_M^{PAW} - E_0  \leq C \left( \eta^{2N} + \frac{\eta^2}{(\eta M)^n} \right).
$$
Hence taking a plane wave cut-off $M \geq \frac{1}{\eta}$ would ensure that the eigenvalue $E_0$ is computed up to an error of order $\cO (\eta^{2N})$.
\end{note}

\section{Proofs}
\label{sec:proof}

\subsection{Useful lemmas}

We introduce some notation used in the below proofs. Let $I \in \lbrace 0, a \rbrace$ and
\begin{align*}
p^I(t) & := (p_1^I(t), \dots, p_{N}^I(t))^T \in \R^N, \\
\widetilde{p}^I(t) & := (\widetilde{p}_1^I(t), \dots, \widetilde{p}_{N}^I(t))^T \in \R^N, \\
\psh{\widetilde{p}^I}{f} & := \left(\psh{\widetilde{p}_1^I}{f}, \dots, \psh{\widetilde{p}_{N}^I}{f}\right)^T \in \R^N, \forall \, f \in L^2_\mathrm{per}(0,1),\\
\Phi_I(t) & := (\phi_1^I(t), \dots, \phi_{N}^I(t))^T \in \R^N, \\
\widetilde{\Phi}_I(t) & := (\widetilde{\phi}_1^I(t), \dots, \widetilde{\phi}_{N}^I(t))^T \in \R^N, \\
A_I & := ( \langle p_i^I, \phi_j^I \rangle)_{1 \leq i,j \leq N} \in \R^{N \times N}.
\end{align*}
For $p \in [1, \infty]$, we denote by 
$$
\| f \|_{p,\eta,I} = \begin{cases}
\|f\|_{L^p(-\eta,\eta)}, & \quad \text{if } I = 0 \\
\|f\|_{L^p(a-\eta,a+\eta)}, & \quad \text{if } I = a
\end{cases}.
$$ 

First, we recall some results of \cite{blanc2017vpaw1d} that are useful for the proofs of Theorems \ref{thm:bound_trunc_PAW} to \ref{thm:E_PAW_M}.  

\begin{lem}
\label{lem:approximation}
Let $\widetilde\psi$ be an eigenfunction of \eqref{eq:H_mol} associated to the lowest eigenvalue $E_0$ and $\widetilde{\psi}_e$ be its even part. Let $\psi = (\mathrm{Id} + T_N) \widetilde{\psi}$ where $T_N$ is the operator \eqref{eq:T_N} and $\psi_e$ be the even part of $\psi$. Suppose $\eta_0 > 0$ satisfies Assumption \ref{assump:1} and Assumption \ref{assump:2}. Then there exists a constant $C$ independent of $\eta$ such that for any $0 < \eta \leq \eta_0$ we have
$$
\left\| \widetilde{\psi}_e - \psh{\widetilde{p}^I}{\widetilde{\psi}}^T \widetilde{\Phi}_I \right\|_{\infty,\eta,I} \leq C \eta^{2N},
$$
and 
$$
\left\| E_0 \psi_e - \psh{A_I^{-1}p^I}{\psi} \cdot \cE^I \Phi_I \right\|_{\infty,\eta,I} \leq C \eta^{2N-2},
$$
where  $\cE^I$ is the $N\times N$ diagonal matrix with entries $(-\epsilon_1^I, \dots, -\epsilon_N^I)$.
\end{lem}

\begin{proof}
We have
\begin{equation}
\label{eq:widetilde_p-Ap-psi}
\widetilde{\psi} - \psh{\widetilde{p}^I}{\widetilde{\psi}}^T \widetilde{\Phi}_I = \psi - \psh{A_I^{-1} p}{\psi} \cdot \Phi_I ,
\end{equation}
and in combination with Lemmas 4.2 and 4.6 in \cite{blanc2017vpaw1d}, we obtain
\begin{equation*}
\left\| \widetilde{\psi}_e - \psh{\widetilde{p}^I}{\widetilde{\psi}}^T \widetilde{\Phi}_I \right\|_{\infty,\eta,I} \leq C \eta^{2N},
\end{equation*}
where $C>0$ is independent of $\eta$.

The second estimate is proved the same way.
\end{proof}

\begin{lem}
\label{lem:matrice_moche}
Let $P_k(t) = \frac{1}{2^k k!} (t^2-1)^k$ and $P(t) = (P_0(t), \dots, P_{d-1}(t))^T$. Let $C^{(P)}_\eta \in \R^{N \times d}$ be the matrix such that for $t \in (-\eta,\eta)$, 
$$
\widetilde{\Phi}_I(t) = C^{(P)}_\eta P(\tfrac{t}{\eta}).
$$
Let $C_1 \in \R^{N \times N}$ and $C_2 \in \R^{N \times (d-N)}$ be the matrices such that
$$
C^{(P)}_\eta = \Big( C_1 \ \Big| \ C_2 \Big). 
$$
Let $M_\eta$ be the matrix
$$
M_\eta = \left(C_\eta^{(P)}\right)^T \left( C_\eta^{(P)} G(P) \left(C_\eta^{(P)}\right)^T \right)^{-1} C_\eta^{(P)},
$$
where $G(P)$ is the matrix $\itg{-1}{1}{\rho(t) P(t)P(t)^T}{t}$.

Then the following statements hold.
\begin{enumerate}
\item the norm of the matrix $M_\eta$ is uniformly bounded in $\eta$.

\item for all $x \in (-\eta,\eta)$
\begin{equation*}
\psh{\widetilde{p}^I}{f}^T  \Phi_I(x)  = \left( M_\eta \itg{-1}{1}{\rho(t) f(\eta t) P(t)}{t} \right)^T  \begin{pmatrix} C_1^{-1} \\ 0 \end{pmatrix} \Phi_I(x) ,
\end{equation*}
and
\begin{equation*}
\psh{\widetilde{p}^I}{f}^T \widetilde{\Phi}_I(x) = \left( M_\eta \itg{-1}{1}{\rho(t) f(\eta t) P(t)}{t} \right)^T  P(x/\eta) .
\end{equation*}

\item for all $0 < \eta \leq \eta_0$ and $x \in (-\eta,\eta)$
$$
C_1^{-1} \Phi_I(x) = \begin{pmatrix}
1 \\ *
\end{pmatrix} + \cO(\eta) \quad \text{and} \quad  C_1^{-1} \Phi_I'(x) = \frac{1}{\eta} \begin{pmatrix}
0 \\ *
\end{pmatrix} + \cO(1),
$$
where $\begin{pmatrix} 1 \\ * \end{pmatrix}$ and $\begin{pmatrix} 0 \\ * \end{pmatrix}$ are uniformly bounded in $\eta$ and $x$.
\end{enumerate}
\end{lem}

\begin{proof}
Proofs of these statements can be found in the proof of Lemma 4.13 and 4.14 in \cite{blanc2017vpaw1d}.
\end{proof}

\begin{lem}
\label{lem:widetilde_p_stuff}
There exists a positive constant $C$ independent of $f$ and $\eta$ such that we have the following estimates
\begin{enumerate}
\item for all $f \in H^1_\mathrm{per}(0,1)$, $0 < \eta \leq \eta_0$ and $x \in (-\tfrac{1}{2},\tfrac{1}{2})$, we have
$$
\left|\psh{\widetilde{p}^I}{f}^T (\Phi_I(x) - \widetilde{\Phi}_I(x)) \right| \leq C \eta \|f\|_{H^1_\mathrm{per}} \quad \text{and} \quad \left|\psh{\widetilde{p}^I}{f}^T (\Phi_I'(x) - \widetilde{\Phi}_I'(x)) \right| \leq C \|f\|_{H^1_\mathrm{per}};
$$
\item for all $f \in L^2_\mathrm{per}(0,1)$, $0 < \eta \leq \eta_0$ and $x \in (-\tfrac{1}{2},\tfrac{1}{2})$, we have
$$
\left|\psh{\widetilde{p}^I}{f}^T (\Phi_I(x) - \widetilde{\Phi}_I(x)) \right| \leq \frac{C}{\eta^{1/2}} \|f\|_{L^2_\mathrm{per}};
$$
\item for all $f \in H^1_\mathrm{per}(0,1)$, $0 < \eta \leq \eta_0$ and $x \in (-\eta,\eta)$, we have
$$
\left|\psh{\widetilde{p}^I}{f}^T \widetilde{\Phi}_I(x) \right| \leq C \|f\|_{L^\infty} \quad \text{and} \quad   \left|\psh{\widetilde{p}^I}{f}^T \Phi_I(x) \right| \leq C \|f\|_{L^\infty};
$$
\item for all $f \in H^1_\mathrm{per}(0,1)$, $0 < \eta \leq \eta_0$ and $x \in (-\eta,\eta)$, we have
$$
\left|\psh{\widetilde{p}^I}{f}^T \widetilde{\Phi}_I'(x) \right| \leq C \|f\|_{H^1_\mathrm{per}} \quad \text{and} \quad  \left|\psh{\widetilde{p}^I}{f}^T \Phi_I'(x) \right| \leq C \|f\|_{H^1_\mathrm{per}}.
$$
\end{enumerate} 
\end{lem}

\begin{proof}
\begin{enumerate}
\item Proof of this statement can be found in \cite{blanc2017vpaw1d} (Lemmas 4.12 and 4.14).
\item By Lemma \ref{lem:matrice_moche}, 
\begin{equation*}
\psh{\widetilde{p}^I}{f}^T \left( \Phi_I(x) - \widetilde{\Phi}_I(x)\right) = \left( M_\eta\itg{-1}{1}{\rho(t) f(\eta t) P(t)}{t} \right)^T \left( \begin{pmatrix} C_1^{-1} \\ 0 \end{pmatrix} \Phi_I(x) - P(x/\eta) \right).
\end{equation*}
Applying the Cauchy-Schwarz inequality to  $\itg{-1}{1}{\rho(t) f(\eta t) P(t)}{t}$ suffices to prove the estimate. 
\item By item 2 of Lemma \ref{lem:matrice_moche}, 
$$
\psh{\widetilde{p}^I}{f}^T \widetilde{\Phi}_I(x) = \left( M_\eta \itg{-1}{1}{\rho(t) f(\eta t) P(t)}{t} \right)^T  P(x/\eta) .
$$
Thus the first inequality follows from the uniform boundedness of $M_\eta$ with respect to $\eta$ (item 1 of Lemma \ref{lem:matrice_moche}). For the second inequality, we proceed the same way and conclude using item 3 of Lemma \ref{lem:matrice_moche}.
\item For the first inequality, we simply replace Step 1 in the proof of Lemma 4.12 in \cite{blanc2017vpaw1d} by 
\begin{enumerate}
\item $\frac{1}{\eta} P'(x/\eta) = \frac{1}{\eta} \begin{pmatrix} 0 \\ * \end{pmatrix} + \cO(1)$
\end{enumerate}
and keep on the proof. For the second inequality, we replace Step 1 in the proof of Lemma 4.12 in \cite{blanc2017vpaw1d} by item 3 of Lemma \ref{lem:matrice_moche}.
\end{enumerate}
\end{proof}

\subsection{PAW method without pseudopotentials}
\label{sec:truncated_PAW}

The main idea of the proof is to use that the PAW operator $H^N$ \eqref{eq:H_N} (respectively $S^N$ \eqref{eq:S_N}) is close to the VPAW operator $\widetilde{H}$ \eqref{eq:H_VPAW} (resp. $\widetilde{S}$ \eqref{eq:S_VPAW}), in a sense that will be clearly stated. Then it is possible to use this connexion and bound the error on the PAW eigenvalue $E^{(\eta)}$, since the VPAW generalized eigenvalue problem \eqref{eq:H_VPAW_eig_pb} has the same eigenvalues as \eqref{eq:H_mol}. 

\begin{prop}
\label{prop:H_N-widetilde-H}
Let $H^N$, $S^N$, $\widetilde{H}$ and $\widetilde{S}$ be defined by Equations (\ref{eq:H_N}), (\ref{eq:S_N}), \eqref{eq:H_VPAW} and \eqref{eq:S_VPAW} respectively. Then we have for $f \in H^1_\mathrm{per}(0,1)$
\begin{equation}
\psh{f}{\widetilde{H}f}  = \psh{f}{H^Nf}  +  2 \sumlim{I=\lbrace 0,a \rbrace}{} \psh{ \psh{\widetilde{p}^I}{f}^T (\Phi_I - \widetilde{\Phi}_I)}{H \left( f - \psh{\widetilde{p}^I}{f}^T \widetilde{\Phi}_I \right)} ,
\end{equation}
and
\begin{equation}
\psh{f}{\widetilde{S}f}  = \psh{f}{S^Nf}  +  2 \sumlim{I=\lbrace 0,a \rbrace}{} \psh{ \psh{\widetilde{p}^I}{f}^T (\Phi_I - \widetilde{\Phi}_I)}{  f -  \psh{\widetilde{p}^I}{f}^T \widetilde{\Phi}_I }  .
\end{equation}
\end{prop}

\begin{proof}
Using that $T_{0,N}$ and $T_{a,N}$ act on strictly distinct region, we have for $f \in H^1_\mathrm{per}(0,1)$
\begin{align*}
\psh{f}{\widetilde{H}  f} & = \psh{f + \sumlim{I=\lbrace 0,a \rbrace}{} \psh{\widetilde{p}^I}{f}^T (\Phi_I - \widetilde{\Phi}_I)}{H \Big( f + \sumlim{I=\lbrace 0,a \rbrace}{} \psh{\widetilde{p}^I}{f}^T (\Phi_I - \widetilde{\Phi}_I) \Big)} \\
& = \psh{f}{Hf} + \sumlim{I=\lbrace 0,a \rbrace}{} 2 \psh{f}{H \psh{\widetilde{p}^I}{f}^T (\Phi_I - \widetilde{\Phi}_I)} \\
& \qquad \qquad \qquad \qquad +  \sumlim{I=\lbrace 0,a \rbrace}{}  \psh{\psh{\widetilde{p}^I}{f}^T (\Phi_I - \widetilde{\Phi}_I)}{H \psh{\widetilde{p}^I}{f}^T (\Phi_I - \widetilde{\Phi}_I)} \\
& = \psh{f}{Hf} + \sumlim{I=\lbrace 0,a \rbrace}{} 2 \psh{f}{H \psh{\widetilde{p}^I}{f}^T (\Phi_I - \widetilde{\Phi}_I)} + \psh{\psh{\widetilde{p}^I}{f}^T \Phi_I}{H \psh{\widetilde{p}^I}{f}^T \Phi_I} \\
& \qquad - 2 \psh{\psh{\widetilde{p}^I}{f}^T \widetilde{\Phi}_I}{H \psh{\widetilde{p}^I}{f}^T \Phi_I} + \psh{\psh{\widetilde{p}^I}{f}^T \widetilde{\Phi}_I}{H \psh{\widetilde{p}^I}{f}^T \widetilde{\Phi}_I} .
\end{align*}
Notice that for each $I$, we have
\begin{equation*}
\begin{split}
\psh{\psh{\widetilde{p}^I}{f}^T \widetilde{\Phi}_I}{H \psh{\widetilde{p}^I}{f}^T \Phi_I} = \psh{\psh{\widetilde{p}^I}{f}^T \right. &  \widetilde{\Phi}_I  \left.}{H \psh{\widetilde{p}^I}{f}^T (\Phi_I - \widetilde{\Phi}_I)} \\
& + \psh{ \psh{\widetilde{p}^I}{f}^T \widetilde{\Phi}_I }{H \psh{\widetilde{p}^I}{f}^T \widetilde{\Phi}_I }.
\end{split}
\end{equation*} 
Hence
\begin{align*}
 \psh{f}{\widetilde{H} f} & =  \psh{f}{Hf} + \sumlim{I=\lbrace 0,a \rbrace}{} \psh{\psh{\widetilde{p}^I}{f}^T \Phi_I}{H \psh{\widetilde{p}^I}{f}^T \Phi_I} - \psh{ \psh{\widetilde{p}^I}{f}^T \widetilde{\Phi}_I }{H \psh{\widetilde{p}^I}{f}^T \widetilde{\Phi}_I } \\
& \qquad + 2 \psh{f - \psh{\widetilde{p}^I}{f}^T \widetilde{\Phi}_I}{H \psh{\widetilde{p}^I}{f}^T (\Phi_I - \widetilde{\Phi}_I)} \\
& = \psh{f}{H^N f} + \sumlim{I=\lbrace 0,a \rbrace}{} 2 \psh{f - \psh{\widetilde{p}^I}{f}^T \widetilde{\Phi}_I}{H \psh{\widetilde{p}^I}{f}^T (\Phi_I - \widetilde{\Phi}_I)}.
\end{align*}
The second identity is proved the same way. 
\end{proof}

Before proving Theorem \ref{thm:E_PAW_bound}, we will state some properties of the operators $\widetilde{S}$ and $S^N$.

\begin{lem}
\label{lem:widetilde_S_L2_bound}
The operators $\widetilde{S}$ and $S^N$ satisfies the following properties
\begin{enumerate}
\item there exists a constant $C$ independent of $\eta$ such that for all $f \in H^1_\mathrm{per}(0,1)$;
$$
\left| \psh{f}{\widetilde{S} f} \right| \leq C \|f\|_{L^2_\mathrm{per}}^2 .
$$
\item there exists a constant $C$ independent of $\eta$ such that for all $f \in H^1_\mathrm{per}(0,1)$;
$$
 | \psh{f}{{S^N} f} | \leq C \|f\|_{L^2_\mathrm{per}}^2.
$$
\item there exists a constant $C$ independent of $\eta$ such that for all $f \in H^1_\mathrm{per}(0,1)$;
$$
\left| \psh{f}{\widetilde{S} f} -  \psh{f}{{S^N} f} \right| \leq C \eta^2 \|f\|_{H^1_\mathrm{per}}^2.
$$
\item let $\widetilde{\psi}$ be a generalized eigenfunction of \eqref{eq:H_VPAW_eig_pb}, then there exists a positive constant $C$ independent of $\eta$ such that
$$
\left| \psh{\widetilde{\psi}}{\widetilde{S} \widetilde{\psi}} -  \psh{\widetilde{\psi}}{{S^N} \widetilde{\psi}} \right| \leq C \eta^{2N+2} \| \widetilde{\psi} \|_{H^1_\mathrm{per}}.
$$
\item there exists a constant $C$ independent of $\eta$ such that for all $f \in H^1_\mathrm{per}(0,1)$;
$$
\left| \psh{f}{S^N f} -  \psh{f}{f} \right| \leq C \eta \|f\|_{H^1_\mathrm{per}}^2.
$$
\end{enumerate}
\end{lem}

\begin{proof}
\begin{enumerate}
\item By item 2 of Lemma \ref{lem:widetilde_p_stuff}, there exists a constant $C$ independent of $\eta$ and $x$ such that for all $x \in (-\tfrac{1}{2},\tfrac{1}{2})$ and for all $0 < \eta \leq \eta_0$
$$
\left|\psh{\widetilde{p}^I}{f}^T (\Phi_I(x) - \widetilde{\Phi}_I(x)) \right| \leq \frac{C}{\eta^{1/2}} \|f\|_{L^2_\mathrm{per}}.
$$
Then, we have
\begin{align*}
\|T_{0,N}f\|_{L^2_\mathrm{per}}^2 & = \itg{0}{1}{\left| \psh{\widetilde{p}}{f}^T (\Phi(x) - \widetilde{\Phi}(x)) \right|^2}{x} \\
& \leq  \itg{-\eta}{\eta}{\left| \psh{\widetilde{p}}{f}^T (\Phi(x) - \widetilde{\Phi}(x)) \right|^2}{x} \\
& \leq C \| f \|_{L^2_\mathrm{per}}^2.
\end{align*}
Similarly, $\|T_{a,N}f\|_{L^2_\mathrm{per}} \leq C \| f \|_{L^2_\mathrm{per}}$ and the result follows.

\item By Proposition \ref{prop:H_N-widetilde-H}, for all $f \in H^1_\mathrm{per}(0,1)$
$$
\psh{f}{S^N f} = \psh{f}{\widetilde{S} f} - 2 \sumlim{I=\lbrace 0,a \rbrace}{} \psh{ \psh{\widetilde{p}^I}{f}^T (\Phi_I - \widetilde{\Phi}_I)}{  f -  \psh{\widetilde{p}^I}{f}^T \widetilde{\Phi}_I }.
$$
From items 1 and 2 of Lemma \ref{lem:matrice_moche}, it is easy to show that there exists a constant $C$ independent of $\eta$ and $x$ such that for all $x \in (-\eta,\eta)$, $0 < \eta \leq \eta_0$ and $f \in H^1_\mathrm{per}(0,1)$
$$
\left|\psh{\widetilde{p}^I}{f}^T \widetilde{\Phi}_I(x) \right| \leq \frac{C}{\eta^{1/2}} \|f\|_{L^2_\mathrm{per}}.
$$
Hence
\begin{align*}
\left| \psh{f}{S^N f} \right| & \leq \left|\psh{f}{\widetilde{S} f} \right| + 2 \sumlim{I \in \lbrace 0, a \rbrace}{} \left|\psh{ \psh{\widetilde{p}^I}{f}^T (\Phi_I - \widetilde{\Phi}_I)}{  f -  \psh{\widetilde{p}^I}{f}^T \widetilde{\Phi}_I} \right| \\
& \leq C  \|f\|_{L^2_\mathrm{per}}^2 + \sumlim{I \in \lbrace 0, a \rbrace}{} \left\|\psh{\widetilde{p}^I}{f}^T (\Phi_I - \widetilde{\Phi}_I) \right\|_{L^2_\mathrm{per}}  ( \|f \|_{L^2_\mathrm{per}} + \left\| \psh{\widetilde{p}^I}{f}^T \widetilde{\Phi}_I \right\|_{2,\eta,I}) \\
& \leq C \|f\|_{L^2_\mathrm{per}}^2.
\end{align*}

\item This is an easy consequence of Proposition \ref{prop:H_N-widetilde-H} and items 2 and 3 of Lemma \ref{lem:widetilde_p_stuff}.

\item By Proposition \ref{prop:H_N-widetilde-H}
$$
\psh{\widetilde{\psi}}{S^N \widetilde{\psi}} = \psh{\widetilde{\psi}}{\widetilde{S} \widetilde{\psi}} - 2 \sumlim{I \in \lbrace 0, a \rbrace}{} \psh{ \psh{\widetilde{p}^I}{\widetilde{\psi}}^T (\Phi_I - \widetilde{\Phi}_I)}{ \widetilde{\psi}  -  \psh{\widetilde{p}^I}{\widetilde{\psi}}^T \widetilde{\Phi}_I }.
$$
By Lemma \ref{lem:approximation}, we have for each $I \in \lbrace 0, a \rbrace$
\begin{equation*}
\left\| \widetilde{\psi} - \psh{\widetilde{p}^I}{\widetilde{\psi}}^T \widetilde{\Phi}_I \right\|_{\infty,\eta,I} \leq C \eta^{2N},
\end{equation*}
where $C>0$ is independent of $\eta$. Hence, using item 1 of Lemma \ref{lem:widetilde_p_stuff},
\begin{align*}
\left| \psh{ \psh{\widetilde{p}^I}{\widetilde{\psi}}^T (\Phi_I - \widetilde{\Phi}_I)}{  \widetilde{\psi} -  \psh{\widetilde{p}^I}{\widetilde{\psi}}^T \widetilde{\Phi}_I } \right| & \leq  \left\| \psh{\widetilde{p}^I}{\widetilde{\psi}}^T (\Phi_I - \widetilde{\Phi}_I) \right\|_{1,\eta,I} \left\| \widetilde{\psi} -  \psh{\widetilde{p}^I}{\widetilde{\psi}}^T \widetilde{\Phi}_I \right\|_{\infty,\eta,I} \\
& \leq C \eta^{2N+2} \| \widetilde{\psi} \|_{H^1_\mathrm{per}}.
\end{align*}
and the result follows. 

\item By item 3 of Lemma \ref{lem:widetilde_S_L2_bound}, we have for all $f \in H^1_\mathrm{per}(0,1)$
$$
\left| \psh{f}{\widetilde{S} f} -  \psh{f}{{S^N} f} \right| \leq C \eta^2 \|f\|_{H^1_\mathrm{per}}^2,
$$
where $C$ is a constant independent of $\eta$ and $f$. 

By item 1 of Lemma \ref{lem:widetilde_p_stuff}, we can easily show that 
$$
\left| \psh{f}{\widetilde{S} f} -  \psh{f}{f} \right| \leq C \eta \|f\|_{H^1_\mathrm{per}}^2, 
$$
with a constant $C$ independent of $\eta$ and $f$. By a triangular inequality, the result follows.
\end{enumerate}
\end{proof}

Before moving to the proof of the upper bound on the PAW eigenvalue \eqref{eq:PAW_eig_pb}, we show that there exists a constant independent of $\eta$ that bounds the $H^1_\mathrm{per}$-norm of $L^2_\mathrm{per}$-normalized  generalized eigenfunctions $\widetilde{\psi}$ associated to the first generalized eigenvalue of $\widetilde{H}$ for all $0 < \eta \leq \eta_0$.

\begin{lem}
\label{lem:H1_boundedness_widetilde_psi}
Let $\widetilde{\psi}$ be an $L^2_\mathrm{per}$-normalized generalized eigenfunction associated to the lowest eigenvalue of \eqref{eq:H_VPAW_eig_pb}. Then there exists a positive constant $C$ independent of $\eta$ such that for all $0 < \eta \leq \eta_0$
$$
\| \widetilde{\psi} \|_{H^1_\mathrm{per}} \leq C.
$$ 
\end{lem}

\begin{proof}
The operator $H$ defined in \eqref{eq:H_mol} is coercive. A proof of this statement can be found in \cite{cances2017discretization}. Let $\alpha > 0$ be such that for all $f \in H^1_\mathrm{per}(0,1)$ 
$$
\psh{f}{H f} + \alpha \psh{f}{f} \geq \frac{1}{2} \|f \|_{H^1_\mathrm{per}}^2. 
$$ 
Then 
\begin{equation*}
\psh{\widetilde{\psi}}{\widetilde{H} \widetilde{\psi}} + \alpha \psh{\widetilde{\psi}}{\widetilde{S} \widetilde{\psi}} \geq \frac{1}{2} \| (\mathrm{Id} + T) \widetilde{\psi} \|_{H^1_\mathrm{per}}^2 .
\end{equation*}
By item 1 of Lemma \ref{lem:widetilde_p_stuff}, we have
$$
\|T\widetilde{\psi}\|_{H^1_\mathrm{per}} \leq C \eta^{1/2} \|\widetilde{\psi}\|_{H^1_\mathrm{per}},
$$
for some positive constant $C$ independent of $\eta$. Hence,
for $\eta$ sufficiently small, there exists a positive constant $C$ independent of $\eta$ such that
$$
(E_0 + \alpha) \psh{\widetilde{\psi}}{\widetilde{S}\widetilde{\psi}}  \geq C \| \widetilde{\psi} \|_{H^1_\mathrm{per}}^2 .
$$
Using item 1 of Lemma \ref{lem:widetilde_S_L2_bound}, we obtain
$$
C \| \widetilde{\psi} \|_{L^2_\mathrm{per}}^2  \geq \| \widetilde{\psi} \|_{H^1_\mathrm{per}}^2,
$$
and the result follows from the normalization of the eigenfunctions $\widetilde{\psi}$. 
\end{proof}

We now have all the necessary tools to prove the upper bound of Theorem \ref{thm:E_PAW_bound}.

\begin{proof}[Proof of the upper bound of Theorem \ref{thm:E_PAW_bound}]
Let $\widetilde{\psi}$ be an $L^2_\mathrm{per}$-normalized eigenvector of the lowest eigenvalue of $\widetilde{H} \widetilde{\psi} = E_0 \widetilde{S} \widetilde{\psi}$. Then by Proposition~\ref{prop:H_N-widetilde-H},
\begin{align*}
\psh{\widetilde{\psi}}{\widetilde{H}\widetilde{\psi}} &  = \psh{\widetilde{\psi}}{H^N \widetilde{\psi}} + 2 \sumlim{I=\lbrace 0,a \rbrace}{} \psh{\widetilde{\psi} - \psh{\widetilde{p}^I}{\widetilde{\psi}}^T \widetilde{\Phi}_I}{ \psh{\widetilde{p}^I}{\widetilde{\psi}}^T H ( \Phi_I - \widetilde{\Phi}_I) } .
\end{align*}
Recall that 
\begin{equation*}
\widetilde{\psi} - \psh{\widetilde{p}^I}{\widetilde{\psi}}^T \widetilde{\Phi}_I = \psi - \psh{\widetilde{p}^I}{\widetilde{\psi}}^T \Phi_I ,
\end{equation*}
which with Equation (\ref{eq:widetilde_p-Ap-psi}) yields
\begin{equation*}
\widetilde{\psi} - \psh{\widetilde{p}^I}{\widetilde{\psi}}^T \widetilde{\Phi}_I = \psi - \psh{A_I^{-1}p^I}{\psi}^T \Phi_I
\end{equation*}
Thus we have :
\begin{align}
\psh{\widetilde{\psi}}{\widetilde{H}\widetilde{\psi}}  & = \psh{\widetilde{\psi}}{H^N \widetilde{\psi}} + 2 \sumlim{I=\lbrace 0,a \rbrace}{} \psh{\psi - \psh{A_I^{-1}p^I}{\psi}^T \Phi_I }{ H \psh{\widetilde{p}^I}{\widetilde{\psi}}^T ( \Phi_I - \widetilde{\Phi}_I) } \nonumber \\
& = \psh{\widetilde{\psi}}{H^N \widetilde{\psi}} + 2 \sumlim{I=\lbrace 0,a \rbrace}{} \psh{ E_0 \psi - \psh{A_I^{-1}p^I}{\psi}^T \cE^I \Phi_I }{ \psh{\widetilde{p}^I}{\widetilde{\psi}}^T ( \Phi_I - \widetilde{\Phi}_I) } ,
\label{eq:E_eta-E_0}
\end{align} 
where we used $H \Phi_I = \cE^I \Phi_I$ in $(I-\eta,I+\eta)$ for $I \in \lbrace	0,a \rbrace$. By Lemma \ref{lem:approximation},
$$
\left\| E_0 \psi_e - \psh{A_I^{-1}p^I}{\psi}^T \cE^I \Phi_I \right\|_{\infty,\eta,I} \leq C\eta^{2N-2}.
$$
So for each $I$,
\begin{multline*}
\left| \psh{ E_0 \psi_e - \psh{A_I^{-1}p^I}{\psi}^T \cE^I \Phi_I }{ \psh{\widetilde{p}^I}{\widetilde{\psi}}^T ( \Phi_I - \widetilde{\Phi}_I) } \right| \\
 \leq \left\|  E_0 \psi_e - \psh{A_I^{-1}p^I}{\psi}^T \cE^I \Phi_I \right \|_{\infty,\eta,I} \left\| \psh{\widetilde{p}^I}{\widetilde{\psi}}^T ( \Phi_I - \widetilde{\Phi}_I) \right\|_{1,\eta,I}.
\end{multline*}
By item 1 of Lemma \ref{lem:widetilde_p_stuff}, we have
\begin{equation*}
\left\| \psh{\widetilde{p}^I}{\widetilde{\psi}}^T ( \Phi_I - \widetilde{\Phi}_I) \right\|_{1,\eta,I} \leq C \eta^2 \|\widetilde{\psi} \|_{H^1_\mathrm{per}} \leq C \eta^2 .
\end{equation*}
where we bound $\|\widetilde{\psi} \|_{H^1_\mathrm{per}}$ by means of Lemma \ref{lem:H1_boundedness_widetilde_psi}. Hence, using Lemma \ref{lem:approximation}, we obtain
\begin{equation*}
\left| \psh{ E_0 \psi - \psh{A_I^{-1}p^I}{\psi}^T \cE^I \Phi_I }{ \psh{\widetilde{p}^I}{\widetilde{\psi}}^T ( \Phi_I - \widetilde{\Phi}_I) } \right| \leq C \eta^{2N}.
\end{equation*}
Going back to Equation (\ref{eq:E_eta-E_0}), 
\begin{align*}
E_0 \psh{\widetilde{\psi}}{\widetilde{S} \widetilde{\psi}} + C \eta^{2N} & \geq \psh{\widetilde{\psi}}{H^N \widetilde{\psi}}   \\
& \geq E^{(\eta)} \psh{\widetilde{\psi}}{S^N \widetilde{\psi}} 
\end{align*}
By Lemmas \ref{lem:widetilde_S_L2_bound} and \ref{lem:H1_boundedness_widetilde_psi}, we have
\begin{equation*}
\left| \psh{\widetilde{\psi}}{\widetilde{S} \widetilde{\psi}} - \psh{\widetilde{\psi}}{S^N \widetilde{\psi}} \right| \leq C \eta^{2N+2},
\end{equation*}
which finishes the proof.
\end{proof}

\begin{lem}
\label{lem:H1_boundedness_f}
Let $f$ be an $L^2_\mathrm{per}$-normalized generalized eigenfunction associated to the lowest generalized eigenvalue of \eqref{eq:PAW_eig_pb}. Then there exists a positive constant $C$ independent of $\eta$ such that for all $0 < \eta \leq \eta_0$
$$
\| f \|_{H^1_\mathrm{per}} \leq C.
$$ 
\end{lem}

\begin{proof}
We proceed as in the proof of Lemma \ref{lem:H1_boundedness_widetilde_psi}. Let $\alpha$ be the coercivity constant of $H$ and $f$ be an $L^2_\mathrm{per}$-normalized eigenfunction associated to the lowest eigenvalue of \eqref{eq:PAW_eig_pb}. Then we have
$$
\alpha \psh{f}{f} + \psh{f}{Hf} \geq \frac{1}{2} \| f \|_{H^1_\mathrm{per}}^2.
$$
From Equation \eqref{eq:H_N}, it easy to see that we have
$$
\psh{f}{H^N f} = \psh{f}{Hf} + \sumlim{I \in \lbrace 0, a\rbrace}{} \psh{\psh{\widetilde{p}^I}{f}^T (\Phi_I + \widetilde{\Phi}_I)}{H \psh{\widetilde{p}^I}{f}^T (\Phi_I - \widetilde{\Phi}_I)} .
$$
Hence, we have
\begin{align*}
\alpha \psh{f}{f} + \psh{f}{H^N f} & -  \sumlim{I \in \lbrace 0, a\rbrace}{} \psh{\psh{\widetilde{p}^I}{f}^T (\Phi_I + \widetilde{\Phi}_I)}{H \psh{\widetilde{p}^I}{f}^T (\Phi_I - \widetilde{\Phi}_I)} \geq \frac{1}{2} \| f \|_{H^1_\mathrm{per}}^2 \\
\alpha \psh{f}{f} + \psh{f}{H^N f} & \geq \frac{1}{2} \| f \|_{H^1_\mathrm{per}}^2 - C \sumlim{I \in \lbrace 0, a\rbrace}{}  \| \psh{\widetilde{p}^I}{f}^T (\Phi_I + \widetilde{\Phi}_I) \|_{H^1,\eta,I} \| \psh{\widetilde{p}^I}{f}^T (\Phi_I - \widetilde{\Phi}_I) \|_{H^1,\eta,I}.
\end{align*}
From items 1, 3 and 4 of Lemma \ref{lem:widetilde_p_stuff}, it is easy to show that 
\begin{equation}
\label{eq:terme_en_plus}
\left\| \psh{\widetilde{p}^I}{f}^T (\Phi_I + \widetilde{\Phi}_I) \right\|_{H^1,\eta,I} \left\| \psh{\widetilde{p}^I}{f}^T (\Phi_I - \widetilde{\Phi}_I) \right\|_{H^1,\eta,I} \leq C \eta \|f\|_{H^1_\mathrm{per}}^2.
\end{equation}
Thus, for $\eta$ sufficiently small, we have for a positive constant $C$ independent of $\eta$,
\begin{equation}
\label{eq:usefule_for_E_PAW_bddness}
\alpha \psh{f}{f} + \psh{f}{H^N f}  \geq C \| f \|_{H^1_\mathrm{per}}^2 .
\end{equation}
Since $f$ is a generalized eigenfunction of $H^N$, we have
$$
\alpha \psh{f}{f} + E^{(\eta)} \psh{f}{S^N f}  \geq C \| f \|_{H^1_\mathrm{per}}^2 .
$$
By item 5 of Lemma \ref{lem:widetilde_S_L2_bound}, we have
$$
 (E^{(\eta)} +\alpha) \psh{f}{f}  \geq C \| f \|_{H^1_\mathrm{per}}^2 ,
$$
which completes the proof. 
\end{proof}

\begin{proof}[Proof of the lower bound of Theorem \ref{thm:E_PAW_bound}]
Let $f$ be an $L^2_\mathrm{per}$-normalized eigenfunction associated to the lowest eigenvalue of $H^N f = E^{(\eta)} S^N f$. Then we have :
\begin{align*}
\psh{f}{H^N f} & = \psh{f}{H f} + \sumlim{I=\lbrace 0,a \rbrace}{}  \psh{\psh{\widetilde{p}^I}{f}^T (\Phi_I + \widetilde{\Phi}_I)}{H \psh{\widetilde{p}^I}{f}^T (\Phi_I - \widetilde{\Phi}_I)} \\
& \geq E_0 \psh{f}{f} - C \sumlim{I=\lbrace 0,a \rbrace}{}  \| \psh{\widetilde{p}^I}{f}^T (\Phi_I + \widetilde{\Phi}_I) \|_{H^1,\eta,I} \| \psh{\widetilde{p}^I}{f}^T (\Phi_I - \widetilde{\Phi}_I) \|_{H^1,\eta,I} \\
& \geq E_0 \psh{f}{f} - C \eta \|f\|_{H^1_\mathrm{per}},
\end{align*}
where we used \eqref{eq:terme_en_plus} in the last inequality.

It remains to show that $|\psh{f}{S^N f} - \psh{f}{f}| \leq C \eta \|f\|_{H^1_\mathrm{per}}^2$ which is precisely item 5 of Lemma \ref{lem:widetilde_S_L2_bound}. We then conclude the proof by Lemma \ref{lem:H1_boundedness_f}.  
\end{proof}

\subsection{PAW method with pseudopotentials}

In this section, we focus on the truncated equations \eqref{eq:PAW_eig_pb_pseudo} where a pseudopotential is used. First, we see how $H^{PAW}$ and $\widetilde{H}$ are related. 

\begin{lem}
\label{prop:H_PAW_link}
If $\epsilon \leq \eta$, then
\begin{equation}
H^{PAW} = H^N + \delta V- \sumlim{I \in \lbrace 0, a \rbrace}{} (\widetilde{p}^I)^T \psh{\widetilde{\Phi}_I}{\delta V\widetilde{\Phi}_I^T}_{I,\eta} \psh{\widetilde{p}^I}{\bm{\cdot}}.
\end{equation}
where $\delta V= - Z_0 \chi_\epsilon - Z_a \chi_\epsilon^a + Z_0 \sumlim{k \in \Z}{} \delta_k + Z_a \sumlim{k \in \Z}{} \delta_{k+a}$. 
\end{lem}

\begin{proof}
By definition of the pseudo wave functions $\widetilde{\phi}_i$, we have 
\begin{equation}
\psh{\phi_i^I}{H \phi_j^I} - \psh{\widetilde{\phi}_i^I}{H \widetilde{\phi}_j^I} = \psh{\phi_i^I}{H \phi_j^I}_{I,\eta} - \psh{\widetilde{\phi}_i^I}{H \widetilde{\phi}_j^I}_{I,\eta}.
\end{equation}
By definition of $\delta V$, $H_\mathrm{ps} = H + \delta V$ thus we have the result. 
\end{proof}

\begin{prop}
\label{prop:H_PAW_formquad}
Let $g \in H^1_{\mathrm{per}}(0,1)$. Then
\begin{multline}
\psh{g}{H^{PAW}g} = \psh{g}{\widetilde{H}g} - 2  \sumlim{I \in \lbrace 0, a \rbrace}{} \psh{g - \psh{\widetilde{p}^I}{g}^T \widetilde{\Phi}_I}{ \psh{\widetilde{p}^I}{g}^T \left(H \Phi_I - (H + \delta V) \widetilde{\Phi}_I\right) } \\
+  \sumlim{I \in \lbrace 0, a \rbrace}{} \psh{g - \psh{\widetilde{p}^I}{g}^T \widetilde{\Phi}_I}{\delta V\left( g - \psh{\widetilde{p}^I}{g}^T \widetilde{\Phi}_I \right)}_{I,\eta}
\end{multline}
\end{prop}

\begin{proof}
By Lemma \ref{prop:H_PAW_link}, we have
\begin{equation*}
\psh{g}{H^{PAW}g} = \psh{g}{H^N g} + \psh{g}{\delta Vg} - \sumlim{I \in \lbrace 0, a \rbrace}{}  \psh{\psh{\widetilde{p}^I}{g}^T \widetilde{\Phi}_I}{\delta V \psh{\widetilde{p}^I}{g}^T \widetilde{\Phi}_I}_{I,\eta} .
\end{equation*}
Applying Proposition \ref{prop:H_N-widetilde-H}, we obtain
\begin{multline*}
\psh{g}{H^{PAW}g} = \psh{g}{\widetilde{H} g} - 2 \sumlim{I \in \lbrace 0, a \rbrace}{} \psh{g - \psh{\widetilde{p}^I}{g}^T \widetilde{\Phi}_I}{ \psh{\widetilde{p}^I}{g}^T H ( \Phi_I - \widetilde{\Phi}_I) }\\
 + \psh{g}{\delta Vg}- \sumlim{I \in \lbrace 0, a \rbrace}{}  \psh{\psh{\widetilde{p}^I}{g}^T \widetilde{\Phi}_I}{\delta V\psh{\widetilde{p}^I}{g}^T \widetilde{\Phi}_I} .
\end{multline*}
Now, using $H_\mathrm{ps} = H + \delta V$, we get
\begin{multline}
\label{eq:H_PAW_1form}
\psh{g}{H^{PAW}g} = \psh{g}{\widetilde{H} g} - 2 \sumlim{I \in \lbrace 0, a \rbrace}{} \psh{g - \psh{\widetilde{p}^I}{g}^T \widetilde{\Phi}_I}{ \psh{\widetilde{p}^I}{g}^T (H \Phi_I - H_\mathrm{ps} \widetilde{\Phi}_I )}\\
 - 2 \sumlim{I \in \lbrace 0, a \rbrace}{} \psh{g - \psh{\widetilde{p}^I}{g}^T \widetilde{\Phi}_I}{\delta V \psh{\widetilde{p}^I}{g}^T \widetilde{\Phi}_I}_{I,\eta}
 + \psh{g}{\delta Vg}-  \sumlim{I \in \lbrace 0, a \rbrace}{}  \psh{\psh{\widetilde{p}^I}{g}^T \widetilde{\Phi}_I}{\delta V\psh{\widetilde{p}^I}{g}^T \widetilde{\Phi}_I}.
\end{multline}
Notice that for each $I$,
\begin{align*}
- 2  & \psh{g - \psh{\widetilde{p}^I}{g}^T \widetilde{\Phi}_I}{\delta V \psh{\widetilde{p}^I}{g}^T \widetilde{\Phi}_I}_{I,\eta}
 + \psh{g}{\delta Vg}_{I,\eta} -  \psh{\psh{\widetilde{p}^I}{g}^T \widetilde{\Phi}_I}{\delta V\psh{\widetilde{p}^I}{g}^T \widetilde{\Phi}_I} \\
  = & \psh{g}{\delta Vg}_{I,\eta} - 2 \psh{g }{ \delta V \psh{\widetilde{p}^I}{g}^T \widetilde{\Phi}_I) }_{I,\eta}  + \psh{\psh{\widetilde{p}^I}{g}^T \widetilde{\Phi}_I}{\delta V \psh{\widetilde{p}^I}{g}^T \widetilde{\Phi}_I }_{I,\eta} \\
 = & \psh{g - \psh{\widetilde{p}^I}{g}^T \widetilde{\Phi}_I}{\delta V\left( g - \psh{\widetilde{p}^I}{g}^T \widetilde{\Phi}_I \right)}_{I,\eta} .
\end{align*}
Injecting this expression in \eqref{eq:H_PAW_1form}, we have the result. 
\end{proof}

\subsubsection{Proof of the upper bound of Theorem \ref{thm:E_PAW_bound}}
\label{subsec:proof_upper_bound_PAW_pseudo}

\begin{proof}[Proof of the upper bound of Theorem \ref{thm:E_PAW_bound}]
We start by estimating $\psh{\widetilde{\psi}}{H^{PAW} \widetilde{\psi}}$ where $\widetilde{\psi}$ is the generalized eigenfunction associated to the lowest eigenvalue: $\widetilde{H} \widetilde{\psi} = E_0 \widetilde{S} \widetilde{\psi}$. Thus we have :
\begin{align*}
\psh{\widetilde{\psi}}{H^{PAW}\widetilde{\psi}} & = \psh{\widetilde{\psi}}{\widetilde{H}\widetilde{\psi}} - 2 \sumlim{I \in \lbrace 0, a \rbrace}{} \psh{\widetilde{\psi} - \psh{\widetilde{p}^I}{\widetilde{\psi}}^T \widetilde{\Phi}_I}{ \psh{\widetilde{p}^I}{\widetilde{\psi}}^T \left(H \Phi_I - (H + \delta V) \widetilde{\Phi}_I\right) } \\
& \qquad + \sumlim{I \in \lbrace 0, a \rbrace}{} \psh{\widetilde{\psi} - \psh{\widetilde{p}^I}{\widetilde{\psi}}^T \widetilde{\Phi}_I}{\delta V\left( \widetilde{\psi} - \psh{\widetilde{p}^I}{\widetilde{\psi}}^T \widetilde{\Phi}_I \right)}_{I,\eta}.
\end{align*}
By Equation \eqref{eq:widetilde_p-Ap-psi}, we have for each $I$
$$
\widetilde{\psi} - \psh{\widetilde{p}^I}{\widetilde{\psi}}^T \widetilde{\Phi}_I = \psi - \psh{A_I^{-1}p^I}{\psi}^T \Phi_I,
$$
so for each $I$
\begin{align*}
\psh{\widetilde{\psi} - \psh{\widetilde{p}^I}{\widetilde{\psi}} \right. & \left. \cdot \widetilde{\Phi}_I}{ \psh{\widetilde{p}^I}{\widetilde{\psi}}^T \left(H \Phi_I - (H + \delta V) \widetilde{\Phi}_I\right) } \\
& = \psh{\psi - \psh{A_I^{-1}p^I}{\psi}^T \Phi_I}{ \psh{\widetilde{p}^I}{\widetilde{\psi}}^T \left(H \Phi_I - (H + \delta V) \widetilde{\Phi}_I\right) } \\
& = \psh{E_0 \psi - \psh{A_I^{-1}p^I}{\psi}^T \cE^I \Phi_I}{ \psh{\widetilde{p}^I}{\widetilde{\psi}}^T (\Phi_I - \widetilde{\Phi}_I) } \\
& \qquad + \left\langle \psi - \psh{A_I^{-1}p^I}{\psi}^T \Phi_I, \psh{\widetilde{p}^I}{\widetilde{\psi}}^T \delta V \widetilde{\Phi}_I \right\rangle_{I,\eta}.
\end{align*}
We have already proved in the proof of the upper bound of Theorem \ref{thm:bound_trunc_PAW} that 
$$
\left| \psh{E_0 \psi - \psh{A_I^{-1}p^I}{\psi}^T \cE^I \Phi_I}{ \psh{\widetilde{p}^I}{\widetilde{\psi}}^T (\Phi_I - \widetilde{\Phi}_I) } \right| \leq C \eta^{2N}.
$$
Moreover by Lemma \ref{lem:approximation} and item 3 of Lemma \ref{lem:widetilde_p_stuff}, we have
\begin{align*}
\left| \psh{\psi - \psh{A_I^{-1}p^I}{\psi}^T \Phi_I}{ \psh{\widetilde{p}^I}{\widetilde{\psi}}^T \delta V \widetilde{\Phi}_I }_{I,\eta}\right| & \leq C  \| \psi_e - \psh{A_I^{-1}p^I}{\psi} \cdot \Phi_I \|_{\infty,\eta,I} \| \psh{\widetilde{p}^I}{\widetilde{\psi}}^T \widetilde{\Phi}_I \|_{\infty,\eta,I} \\
& \leq C \eta^{2N}.
\end{align*}
Again using Lemma \ref{lem:approximation}, we obtain 
\begin{equation}
\left| \psh{\widetilde{\psi} - \psh{\widetilde{p}^I}{\widetilde{\psi}}^T \widetilde{\Phi}_I}{\delta V\left( \widetilde{\psi} - \psh{\widetilde{p}^I}{\widetilde{\psi}}^T \widetilde{\Phi}_I \right)}_{I,\eta} \right| \leq C \eta^{2N}  + \itg{-\eta}{\eta}{\chi_\epsilon(x) |\psi_o(x)|^2}{x},
\end{equation}
where $\psi_o$ is the odd part of $\psi$. By Lemma 4.2 in \cite{blanc2017vpaw1d}, we know that for $|x|\leq \eta$, there exists a constant independent of $\eta$ such that:
$$
|\psi_o(x)|^2 \leq C \eta^2,
$$
hence 
\begin{equation*}
\left| \psh{\widetilde{\psi} - \psh{\widetilde{p}^I}{\widetilde{\psi}}^T \widetilde{\Phi}_I}{\delta V\left( \widetilde{\psi} - \psh{\widetilde{p}^I}{\widetilde{\psi}}^T \widetilde{\Phi}_I \right)}_{I,\eta} \right| \leq C \eta^2.
\end{equation*}
Thus
$$
E^{PAW} \psh{\widetilde{\psi}}{S^{PAW}\widetilde{\psi}} \leq E_0 \psh{\widetilde{\psi}}{\widetilde{S} \widetilde{\psi}} + C \eta^2,
$$
and we conclude using item 4 of Lemma \ref{lem:widetilde_S_L2_bound} (recall that $S^{PAW} = S^N$).
\end{proof}

\subsubsection{Proof of the lower bound of Theorem \ref{thm:E_PAW_bound}}

The core of the proof of the error on the lowest PAW eigenvalue lies on the estimation of \linebreak $f - \sumlim{i=1}{N} \psh{\widetilde{p}_i}{f} \widetilde{\phi}_i$, which is of the order of the best approximation of $f$ by the family of pseudo wave functions $(\widetilde{\phi}_i)_{1 \leq i \leq N}$. In order to give estimates of the best approximation, we analyze the behavior of the PAW eigenfunction $f$, but first, we need an estimate on the PAW eigenvalue.

\begin{lem}
\label{lem:E_PAW_bounded_below}
Let $E^{PAW}$ be the lowest generalized eigenvalue of \eqref{eq:PAW_eig_pb_pseudo}. Then as $\eta$ goes to 0, $E^{PAW}$ is bounded by below. 
\end{lem}

\begin{proof}
Let $f$ be an $L^2_\mathrm{per}$-normalized generalized eigenfunction of \eqref{eq:PAW_eig_pb_pseudo} associated to $E^{PAW}$. By \eqref{eq:usefule_for_E_PAW_bddness}, we have
$$
\alpha \psh{f}{f} + \psh{f}{H^Nf} \geq C \|f\|_{H^1_\mathrm{per}}^2,
$$
where $C$ is some positive constant, $\alpha$ the coercivity constant of $H$ \eqref{eq:H_mol} and $H^N$ the truncated PAW operator \eqref{eq:H_N}. By Lemma \ref{prop:H_PAW_link}, we have 
\begin{equation}
\label{eq:E_PAW_intermediate}
\alpha \psh{f}{f} + \psh{f}{H^{PAW}f} \geq C \|f\|_{H^1_\mathrm{per}} - \psh{f}{\delta V f} + \sumlim{I \in \lbrace 0, a \rbrace}{} \psh{\psh{\widetilde{p}^I}{f}^T \widetilde{\Phi}_I}{\delta V \psh{\widetilde{p}^I}{f}^T \widetilde{\Phi}_I}.
\end{equation}
We have 
\begin{align}
\left| \psh{f}{\delta V f}_{0,\eta} \right| & \leq Z_0 \left| \itg{-\epsilon}{\epsilon}{\chi_\epsilon(x) (|f(x)|^2 - |f(0)|^2)}{x} \right| \nonumber \\
& \leq C \itg{-\epsilon}{\epsilon}{\chi_\epsilon(x) |f(x)+f(0)||f(x)-f(0)|}{x} \nonumber \\
& \leq C \|f\|_{\infty,\eta} \|f - f(0)\|_{\infty,\eta} \nonumber \\
& \leq C \eta^{1/2} \|f\|^2_{H^1_\mathrm{per}} \label{eq:borne_f_dV_f} ,
\end{align}
where in the second inequality, we used $\itg{-\epsilon}{\epsilon}{\chi_\epsilon(x)}{x} = 1$ and $\epsilon \leq \eta$ and in the last inequality, $\|f - f(0)\|_{\infty,\eta} \leq C \eta^{1/2} \|f\|_{H^1_\mathrm{per}}$ and the Sobolev embedding $\|f\|_{L^\infty} \leq C \|f\|_{H^1_\mathrm{per}}$. 

Similarly, we have
\begin{align*}
\left| \psh{\psh{\widetilde{p}^I}{f}^T \widetilde{\Phi}_I}{\delta V \psh{\widetilde{p}^I}{f}^T \widetilde{\Phi}_I} \right| & \leq C \eta^{1/2} \|\psh{\widetilde{p}^I}{f}^T \widetilde{\Phi}_I\|_{H^1_\mathrm{per}}^2,
\end{align*}
thus by items 3 and 4 of Lemma \ref{lem:widetilde_p_stuff}, we obtain
\begin{equation}
\label{eq:borne_tdpPhi_dV_tdpPhi}
\left| \psh{\psh{\widetilde{p}^I}{f}^T \widetilde{\Phi}_I}{\delta V \psh{\widetilde{p}^I}{f}^T \widetilde{\Phi}_I} \right|  \leq C \eta \|f\|_{H^1_\mathrm{per}}^2.
\end{equation}

Thus injecting \eqref{eq:borne_f_dV_f} and \eqref{eq:borne_tdpPhi_dV_tdpPhi} in \eqref{eq:E_PAW_intermediate}, we get for $\eta$ sufficiently small and a positive constant $C$,
$$
\alpha \psh{f}{f} + \psh{f}{H^{PAW}f} \geq C \|f\|_{H^1_\mathrm{per}}^2.
$$
Thus
\begin{equation}
\label{eq:f_PAW_bounded}
\alpha \psh{f}{f} + E^{PAW} \psh{f}{S^{PAW}f} \geq C \|f\|_{H^1_\mathrm{per}}^2,
\end{equation}
and we conclude the proof using item 5 of Lemma \ref{lem:widetilde_S_L2_bound}.
\end{proof}

\begin{lem}
\label{lem:der_f_paw}
Let $f$ be a generalized eigenfunction of \eqref{eq:PAW_eig_pb_pseudo} and $k \in \N^*$. Then there exists a constant $C$ independent of $\eta$, $\epsilon$ and $f$ such that
\begin{equation}
\|f^{(k)} \|_{\infty,\eta,I} \leq C \left( \frac{1}{\eta^{k-1}} + \frac{1}{\epsilon^{k-1}} \right) \|f\|_{\infty,\eta,I}
\end{equation}
\end{lem}

\begin{proof}
This lemma is proved by iteration. We show the lemma for $I=0$ and drop the index $I$. 
\paragraph{Initialization}
To get the desired estimate for $f'$, we integrate \eqref{eq:PAW_eig_pb_pseudo} on $(-\eta,x)$ where $x \in (-\eta,\eta)$:
\begin{multline}
\label{eq:fpaw_der2}
- f''(x) + \frac{1}{\epsilon} \chi\left( \tfrac{x}{\epsilon} \right)f(x) + \psh{\widetilde{p}}{f}^T \left( \psh{\Phi}{H \Phi^T}_\eta - \psh{\widetilde{\Phi}}{H_\mathrm{ps} \widetilde{\Phi}^T}_\eta \right) \widetilde{p}^I(x) \\
= E^{PAW} \left(f(x) + \psh{\widetilde{p}}{f}^T \left( \psh{\Phi}{ \Phi^T}_\eta - \psh{\widetilde{\Phi}}{\widetilde{\Phi}^T}_\eta \right) \widetilde{p}(x) \right).
\end{multline}
First, we bound $f'(\pm \eta)$ and $f'(a \pm \eta)$. For $x \in \bigcup\limits_{k \in \Z}(\eta + k, a-\eta+k)$ and $x \in \bigcup\limits_{k \in \Z}(a+\eta + k, 1-\eta+k)$, $f$ satisfies
$$
-f''(x) = E^{PAW} f(x).
$$
From Section \ref{subsec:proof_upper_bound_PAW_pseudo}, we already know that 
$$
E^{PAW} \leq E_0 + C \eta^2.
$$
Since $E_0 < 0$, then for $\eta$ sufficiently small, $E^{PAW} < 0$. Thus, outside the intervals $(-\eta,\eta)$ and $(a-\eta,a+\eta)$, $f$ can be written as
$$
f(x) = a_1 \cosh\left(\sqrt{-E^{PAW}}x\right) + a_2 \sinh\left(\sqrt{-E^{PAW}}x\right).
$$
The coefficients $a_1$ and $a_2$ are determined by the continuity of $f$ at $\pm \eta$ and $a \pm \eta$. By Lemma~\ref{lem:E_PAW_bounded_below}, $E^{PAW}$ is bounded from below as $\eta$ goes to 0, hence $|f'(\pm \eta)|$ and $f'(a\pm \eta)$ are uniformly bounded with respect to $\eta$ as $\eta$ goes to 0. 
\newline

We now prove that $f'(x)$ is uniformly bounded with respect to $\eta$ and $\epsilon$ as $\eta, \epsilon \rightarrow 0$ for $x \in \bigcup\limits_{k \in \Z} (-\eta+k, \eta+k)$ and $x \in \bigcup\limits_{k \in \Z} (a-\eta+k, a+\eta+k)$.  $\chi \left( \tfrac{\cdot}{\epsilon} \right)$ is a bounded function supported in $(-\epsilon,\epsilon)$, we have
\begin{align*}
\left| \frac{1}{\epsilon} \itg{-\eta}{x}{\chi \left( \tfrac{t}{\epsilon} \right) f(t)}{t} \right| \leq C \|f\|_{\infty,\eta}.
\end{align*}

To finish the proof, it suffices to show that the remaining terms are at most of order $\cO\left(\tfrac{\|f\|_{\infty,\eta}}{\eta}\right)$ with respect to the $\infty$-norm.  These terms will be treated separately.
\begin{enumerate}
\item For $\psh{\widetilde{p}}{f}^T \psh{\Phi}{\Phi^T}_\eta \, \widetilde{p}(x)$, by item 2 of Lemma \ref{lem:matrice_moche}, we have
\begin{multline*}
\psh{\widetilde{p}}{f}^T \psh{\Phi}{\Phi^T}_\eta \, \widetilde{p}(x)  = 
\left( M_\eta \itg{-1}{1}{\rho(t) f(\eta t) P(t)}{t} \right)^T \\
\psh{\left( \begin{array}{c} C_1^{-1} \\ \hline 0 \end{array} \right) \Phi}{ \Phi^T \Big( C_1^{-T} \, \Big| \, 0 \Big) }_\eta  M_\eta \rho\left(\tfrac{x}{\eta}\right) P \left(\tfrac{x}{\eta}\right).
\end{multline*}
According to item 3 of Lemma \ref{lem:matrice_moche}, we already know that 
$$
\left\| \left( \begin{array}{c} C_1^{-1} \\ \hline 0 \end{array} \right) \Phi \right\|_{\infty,\eta} \leq C,
$$
thus
\begin{equation*}
\left|\psh{\widetilde{p}}{f}^T \psh{\Phi}{\Phi^T}_\eta \, \widetilde{p}(x) \right| \leq C \|f\|_{\infty,\eta} \left|\rho\left(\tfrac{x}{\eta}\right) P \left(\tfrac{x}{\eta}\right) \right|.
\end{equation*}

\item Using item 2 of Lemma \ref{lem:matrice_moche}, the term $\psh{\widetilde{p}}{f}^T \psh{\widetilde{\Phi}}{\widetilde{\Phi}^T}_\eta \, \widetilde{p}(x)$ can be written as
\begin{align*}
\psh{\widetilde{p}}{f}^T \psh{\widetilde{\Phi}}{\widetilde{\Phi}^T}_\eta \, \widetilde{p}(x) & = 
\left( M_\eta \itg{-1}{1}{\rho(t) f(\eta t) P(t)}{t}  \right)^T \psh{P\left( \tfrac{\cdot}{\eta} \right) }{P^T\left( \tfrac{\cdot}{\eta} \right) }_\eta \, M_\eta \rho\left(\tfrac{x}{\eta}\right) P \left(\tfrac{x}{\eta}\right).
\end{align*}
Hence, we obtain 
\begin{equation*}
\left| \psh{\widetilde{p}}{f}^T \psh{\widetilde{\Phi}}{\widetilde{\Phi}^T}_\eta \, \widetilde{p}^I(x) \right| \leq C \|f\|_{\infty,\eta} \left|\rho\left(\tfrac{x}{\eta}\right) P \left(\tfrac{x}{\eta}\right) \right|.
\end{equation*}

\item On the LHS of \eqref{eq:fpaw_der2}, the term $\psh{\widetilde{p}}{f}^T \psh{\Phi}{H \Phi^T}_\eta \, \widetilde{p}(x) $ is given by
$$
\psh{\widetilde{p}}{f}^T \psh{\Phi}{H \Phi^T}_\eta \, \widetilde{p}(x) = \psh{\widetilde{p}}{f}^T \psh{\Phi'}{{\Phi'}^T}_\eta \, \widetilde{p}(x) - Z_0 \psh{\widetilde{p}}{f}^T \Phi(0) \Phi(0)^T \widetilde{p}(x).
$$
Like in item 1 above, we can show that 
\begin{equation}
\label{eq:delta0_PAW}
\left| \psh{\widetilde{p}}{f}^T \Phi(0) \Phi(0)^T \widetilde{p}(x) \right| \leq C \|f\|_{\infty,\eta} \left|\rho\left(\tfrac{x}{\eta}\right) P \left(\tfrac{x}{\eta}\right) \right|.
\end{equation}
Using item 3 of Lemma \ref{lem:matrice_moche},
$$
\left\| \left( \begin{array}{c} C_1^{-1} \\ \hline 0 \end{array} \right) \Phi' \right\|_{\infty,\eta} \leq \frac{C}{\eta},
$$
we get 
\begin{equation}
\label{eq:Phi_prime_PAW}
\left|  \psh{\widetilde{p}}{f}^T \psh{\Phi'}{{\Phi'}^T}_\eta \, \widetilde{p}(x) \right| \leq \frac{C}{\eta} \|f\|_{\infty,\eta} \left|\rho\left(\tfrac{x}{\eta}\right) P \left(\tfrac{x}{\eta}\right) \right|.
\end{equation}

\item Finally, for $\psh{\widetilde{p}}{f}^T \psh{\widetilde{\Phi}}{H_\mathrm{ps} \widetilde{\Phi}^T}_\eta \, \widetilde{p}(x)$, we have
$$
\psh{\widetilde{p}}{f}^T \psh{\widetilde{\Phi}}{H_\mathrm{ps} \widetilde{\Phi}^T}_\eta \, \widetilde{p}(x) = \psh{\widetilde{p}}{f}^T \psh{\widetilde{\Phi}'}{\widetilde{\Phi}^\prime{}^T}_\eta \, \widetilde{p}(x) - \frac{Z_0}{\epsilon} \psh{\widetilde{p}}{f}^T  \itg{-\epsilon}{\epsilon}{\chi \left( \tfrac{t}{\epsilon} \right) P( \tfrac{t}{\eta} ) P( \tfrac{t}{\eta} )^T }{t} \ \widetilde{p}(x).
$$
Since $ \epsilon \leq \eta$, $\left| \itg{-\epsilon}{\epsilon}{\chi \left( \tfrac{t}{\epsilon} \right) P( \tfrac{t}{\eta} ) P( \tfrac{t}{\eta} )^T }{t} \right| \leq C \epsilon$ where $C$ is independent of $\eta$ and $\epsilon$. 
Moreover, 
$$
\psh{\widetilde{p}}{f}^T \psh{\widetilde{\Phi}^\prime}{\widetilde{\Phi}^\prime{}^T}_\eta \, \widetilde{p}(x) = \frac{1}{\eta^2} \left( M_\eta \itg{-1}{1}{\rho(t) f(\eta t) P(t)}{t} \right)^T \psh{ P' ( \tfrac{\cdot}{\eta} ) }{P' ( \tfrac{\cdot}{\eta} )^T }_\eta \, M_\eta \rho\left(\tfrac{x}{\eta}\right) P \left(\tfrac{x}{\eta}\right),
$$
hence 
\begin{equation*}
\left| \psh{\widetilde{p}}{f}^T \psh{\widetilde{\Phi}^\prime}{\widetilde{\Phi}^\prime{}^T}_\eta \, \widetilde{p}(x) \right| \leq \frac{C}{\eta} \|f\|_{\infty,\eta} \left|\rho\left(\tfrac{x}{\eta}\right) P \left(\tfrac{x}{\eta}\right) \right|.
\end{equation*}
\end{enumerate}

\paragraph{Iteration}
Suppose the statement is true for any $k \leq n$. We derivate \eqref{eq:fpaw_der2} $(n-1)$ times
\begin{multline}
\label{eq:fpaw_dern}
- f^{(n+1)}(x) + \frac{1}{\epsilon} \left( \chi\left( \tfrac{\cdot}{\epsilon} \right)f \right)^{(n-1)}(x) + \psh{\widetilde{p}}{f}^T \left( \psh{\Phi}{H \Phi^T}_\eta - \psh{\widetilde{\Phi}}{H_\mathrm{ps} \widetilde{\Phi}^T}_\eta \right) \widetilde{p}^{(n-1)}(x) \\
= E^{PAW} \left(f^{(n-1)}(x) + \psh{\widetilde{p}}{f}^T \left( \psh{\Phi}{ \Phi^T}_\eta - \psh{\widetilde{\Phi}}{\widetilde{\Phi}^T}_\eta \right) \widetilde{p}^{(n-1)}(x) \right).
\end{multline}
By the induction hypothesis and since $\epsilon \leq \eta$, we have
\begin{equation}
\label{eq:iter1}
\left| \frac{1}{\epsilon} \left(\chi\left( \tfrac{\cdot}{\epsilon} \right) f \right)^{(n-1)}(x) \right| \leq C \left( \frac{\|f\|_{\infty,\eta}}{\epsilon^n} + \sumlim{k=1}{n-1} \frac{\|f^{(k)}\|_{\infty,\eta}}{\epsilon^{n-k}} \right) \leq C \frac{\|f\|_{\infty,\eta}}{\epsilon^{n}}. 
\end{equation}
We simply give an estimate of the term 
$$
\psh{\widetilde{p}}{f}^T \psh{\Phi}{H \Phi^T}_\eta \widetilde{p}^{(n-1)}(x),
$$
since the other terms appearing in \eqref{eq:fpaw_dern} can be treated the same way. By \eqref{eq:delta0_PAW}, we already know that
$$
\left| \psh{\widetilde{p}}{f}^T \Phi(0) \Phi(0)^T \widetilde{p}^{(n-1)}(x) \right| \leq \frac{C}{\eta^{n-1}} \|f\|_{\infty,\eta} \left| (\rho P)^{(n-1)}(\tfrac{x}{\eta}) \right| \leq \frac{C}{\eta^{n-1}} \|f\|_{\infty,\eta}.
$$
By \eqref{eq:Phi_prime_PAW}, we have
\begin{equation}
\label{eq:iter2}
\left|  \psh{\widetilde{p}}{f}^T \psh{\Phi'}{{\Phi'}^T}_\eta \, \widetilde{p}^{(n-1)}(x) \right| \leq \frac{C}{\eta^n} \|f\|_{\infty,\eta} \left| (\rho P)^{(n-1)}(\tfrac{x}{\eta}) \right| \leq \frac{C}{\eta^{n}} \|f\|_{\infty,\eta}.
\end{equation}
Injecting \eqref{eq:iter1} and \eqref{eq:iter2} in \eqref{eq:fpaw_dern} finishes the proof.
\end{proof}

First, an estimation of the best approximation by $(\widetilde{\phi}_i)_{1 \leq i \leq N}$ of the even part $f_e$ of the PAW eigenfunction $f$ is proved. 

\begin{lem}
\label{lem:f_paw_approx}
Let $f$ be an eigenfunction associated to the lowest eigenvalue of \eqref{eq:PAW_eig_pb_pseudo} and let $f_e$ be the even part of $f$. Suppose that $\epsilon \leq \eta$. Then there exists a family of coefficients $(\alpha_i)_{1 \leq i \leq N}$ and $C$ independent of $\eta$ and $\epsilon$ such that 
\begin{equation*}
\left\| f_e - \sumlim{i=1}{N} \alpha_i \widetilde{\phi}_i^I \right\|_{\infty,\eta,I} \leq C \eta \left( \frac{\eta}{\epsilon} \right)^{2N-1} \| f \|_{\infty,\eta,I},
\end{equation*}
and for the same family of coefficients 
$$
\left\| f_e' - \sumlim{i=1}{N} \alpha_i \widetilde{\phi}_i^I{}^\prime \right\|_{\infty,\eta,I} \leq C \left( \frac{\eta}{\epsilon} \right)^{2N} \| f \|_{\infty,\eta,I}.
$$
\end{lem}

\begin{proof} For clarity, we will drop the index $I$ in this proof. First we write the Taylor expansion of $f$ around $0$, for $|x|\leq \eta$ :
\begin{equation*}
f_e(x) = \sumlim{k=0}{N-1} \frac{f^{(2k)}(0)}{(2k)!} x^{2k} + R_{2N}(f)(x) ,
\end{equation*}
where $R_{2N}(f)$ is the integral form of the remainder
\begin{equation*}
R_{2N}(f)(x) = \itg{0}{x}{\frac{f^{(2N)}(t)}{(2N-1)!} (x-t)^{2N-1}}{t}.
\end{equation*}
The remainder $R_{2N}(f)$ satisfies 
\begin{align*}
|R_{2N}(f)(x)| & \leq C \eta^{2N} \left\|f^{(2N)}\right\|_{\infty,\eta} \\
& \leq C \eta\left( \tfrac{\eta}{\epsilon} \right)^{2N-1} \|f\|_{\infty,\eta},
\end{align*}
where, in the second inequality, we used Lemma \ref{lem:der_f_paw}. Thus, the best approximation of $f$ by a linear combination of $(\widetilde{\phi}_k)_{1 \leq k \leq N}$ is at most of order $\eta$. In the remainder of the proof, we will show that this order is attainable. Setting $t = \tfrac{x}{\eta}$, we obtain 
\begin{equation*}
f_e(x) - \sumlim{i=1}{N} \alpha_i \widetilde{\phi}_i(x) = \sumlim{k=0}{N-1} \frac{f^{(2k)}(0)}{(2k)!} \eta^{2k} t^{2k} - \sumlim{i=1}{N} \alpha_i \widetilde{\phi}_i(\eta t) + R_{2N}(f)(\eta t).
\end{equation*}
By Lemma \ref{lem:der_f_paw}, we have for all $1 \leq k \leq N-1$:
\begin{equation*}
\left| \frac{f^{(2k)}(0)}{(2k)!} \eta^{2k} \right| \leq C \eta \left( \frac{\eta}{\epsilon} \right)^{2k-1}. 
\end{equation*}
The family $(\widetilde{\phi}_j)_{1 \leq j \leq N}$ satisfies
\begin{equation*}
\widetilde{\Phi}(x) = C_\eta^{(P)} P(\tfrac{x}{\eta}),
\end{equation*}
where $P(t)$ is the vector of polynomials $P_k(t) = \frac{1}{2^k k!}(t^2-1)^{k}$.
By Lemma 4.9 in \cite{blanc2017vpaw1d}, we know that $C_\eta^{(P)}$ can be written:
\begin{equation}
\label{eq:C_eta}
C_\eta^{(P)} = \Phi(\eta)e_0^T + \eta \Phi'(\eta) \beta_1^T+ \cO(\eta^2),
\end{equation}
where $\beta_1$ is a vector of $\R^d$ uniformly bounded in $\eta$. Thus we have 
\begin{equation*}
\sumlim{k=0}{N-1} \frac{f^{(2k)}(0)}{(2k)!} \eta^{2k} t^{2k} - \sumlim{i=1}{N} \alpha_i \widetilde{\phi}_i(\eta t) = f(0) - \alpha^T \Phi(\eta) + \cO\left( \eta \left( \frac{\eta}{\epsilon} \right)^{2N-1} \right). 
\end{equation*}
To get the result, $\alpha$ has to be chosen such that $\alpha^T \Phi(\eta) = f(0)$, which is possible because $\Phi(\eta) \not= 0$.

For $f'_e$, we proceed the same way. However, by Lemma \ref{lem:der_f_paw}, the remainder of the Taylor expansion of $f'_e$ satisfies
$$
|R_{2N}(f')(x)| \leq C \eta^{2N} \left\|f^{(2N+1)}\right\|_{\infty,\eta} \leq C \left( \tfrac{\eta}{\epsilon} \right)^{2N} \|f\|_{\infty,\eta}.
$$ 
We simply have to check that $\| \widetilde{\Phi}' \|_{\infty, \eta}$ is bounded when $\eta$ goes to 0. By \eqref{eq:C_eta} and because $P_0' = 0$, 
$$
\widetilde{\Phi}'(x) = \Phi'(\eta) \beta_1^T P'(\tfrac{x}{\eta})+ \cO(\eta),
$$
hence  $\| \widetilde{\Phi}' \|_{\infty, \eta}$ is bounded when $\eta$ goes to 0.
\end{proof}

We can now give an estimate for $f_e - \sumlim{i=1}{N} \psh{\widetilde{p}}{f} \phi_i$.

\begin{lem}
\label{lem:f_paw_approx2}
Assume that $f$ is the generalized eigenfunction of \eqref{eq:PAW_eig_pb_pseudo} associated the lowest generalized eigenvalue. Let $f_e$ be the even part of $f$. Then 
\begin{equation*}
\left\| f_e - \psh{\widetilde{p}^I}{f}^T \widetilde{\Phi}_I \right\|_{\infty,\eta,I} \leq C \eta \left( \frac{\eta}{\epsilon} \right)^{2N-1} \|f\|_{\infty,\eta,I}, 
\end{equation*}
and 
\begin{equation*}
\left\| f_e' - \psh{\widetilde{p}^I}{f}^T \widetilde{\Phi}_I' \right\|_{\infty,\eta,I} \leq C \left( \frac{\eta}{\epsilon} \right)^{2N-1} \|f\|_{\infty,\eta,I}.
\end{equation*}
\end{lem}

\begin{proof}
For clarity, we will drop the index $I$. For any family $(\alpha_j)_{1 \leq j \leq N}$, we have for $x \in (-\eta,\eta)$
\begin{align*}
f_e(x) - \psh{\widetilde{p}}{f}^T \widetilde{\Phi}(x) & = f_e(x) - \psh{\widetilde{p}}{f_e - \sumlim{j=1}{N} \alpha_j \widetilde{\phi}_j + \sumlim{j=1}{N} \alpha_j \widetilde{\phi}_j}^T \widetilde{\Phi}(x) \\
& = f_e(x) - \sumlim{j=1}{N} \alpha_j \widetilde{\phi}_j - \psh{\widetilde{p}}{f_e - \sumlim{j=1}{N} \alpha_j \widetilde{\phi}_j}^T \widetilde{\Phi}(x).
\end{align*}
By Lemma \ref{lem:f_paw_approx},  $(\alpha_j)_{1 \leq j \leq N}$ can be chosen such that for any $x \in (-\eta,\eta)$
$$
\left| f_e(x) - \sumlim{j=1}{N} \alpha_j \widetilde{\phi}_j(x) \right| \leq C \eta \left( \frac{\eta}{\epsilon} \right)^{2N-1} \|f\|_{\infty,\eta}.
$$
Thus by item 3 of Lemma \ref{lem:widetilde_p_stuff}, 
\begin{equation*}
\left\| f_e - \psh{\widetilde{p}}{f}^T \widetilde{\Phi} \right\|_{\infty,\eta} \leq C \eta \left( \frac{\eta}{\epsilon} \right)^{2N-1} \|f\|_{\infty,\eta}.
\end{equation*}

Similarly, we have by item4 of Lemma \ref{lem:widetilde_p_stuff} for any function $g \in H^1_\mathrm{per}(0,1)$ with $g' \in L^\infty(-\eta,\eta)$,
\begin{equation}
\label{eq:tdpPhi_raffined}
\left| \psh{\widetilde{p}}{g}^T \widetilde{\Phi}'(x) \right| \leq C \|g\|_{H^1,\eta} \leq C \eta^{1/2}(\|g\|_{\infty,\eta} + \|g'\|_{\infty,\eta}),
\end{equation}
and with the same coefficients $(\alpha_j)$, 
$$
\left| f_e'(x) - \sumlim{j=1}{N} \alpha_j \widetilde{\phi}'_j \right|\leq  C \left( \frac{\eta}{\epsilon} \right)^{2N-1} \|f\|_{\infty,\eta}.
$$
So,
\begin{align*}
\left\| f_e' - \psh{\widetilde{p}}{f}^T \widetilde{\Phi}' \right\|_{\infty,\eta} & \leq  \left\| f_e' - \sumlim{j=1}{N} \alpha_j \widetilde{\phi}'_j \right\|_{\infty,\eta} + \left\|  \psh{\widetilde{p}}{f_e - \sumlim{j=1}{N} \alpha_j \widetilde{\phi}_j}^T \widetilde{\Phi}' \right\|_{\infty,\eta} \\
& \leq C  \left( \frac{\eta}{\epsilon} \right)^{2N-1} \|f\|_{\infty,\eta},
\end{align*}
where in the last inequality, we used \eqref{eq:tdpPhi_raffined} with Lemma \ref{lem:f_paw_approx}.
\end{proof}

In the proof of the lower bound of Theorem \ref{thm:E_PAW_bound}, we will need to bound terms of the form $\left\| f_e - \psh{\widetilde{p}^I}{f}^T \widetilde{\Phi}_I \right\|_{\infty, \eta, I}$. If $\epsilon < \eta$, we will get worse bounds than by setting $\epsilon = \eta$. Hence, from now on, we fix $\epsilon = \eta$. 

To estimate the term $\psh{f - \psh{\widetilde{p}^I}{f}^T \widetilde{\Phi}_I}{\delta V\left( f - \psh{\widetilde{p}^I}{f}^T \widetilde{\Phi}_I \right)}_{I,\eta}$, we will need the following estimates.

\begin{lem}
\label{lem:f_paw_approx_DeltaV}
Let $f$ be an eigenfunction associated to the lowest generalized eigenvalue of \eqref{eq:PAW_eig_pb_pseudo}. Then 
\begin{equation*}
\left\| f - \psh{\widetilde{p}^I}{f}^T \widetilde{\Phi}_I \right\|_{\infty,\eta,I} \leq C \eta  \|f\|_{\infty,\eta,I}, 
\end{equation*}
and 
\begin{equation*}
\left\| f' - \psh{\widetilde{p}^I}{f}^T \widetilde{\Phi}_I' \right\|_{\infty,\eta,I} \leq C \|f\|_{\infty,\eta,I}, 
\end{equation*}
\end{lem}

\begin{proof}
This follows from Lemma \ref{lem:f_paw_approx2} and that the odd part of $f$ is bounded  in $(-\eta, \eta)$ by $\eta \|f'\|_{L^\infty(-\eta,\eta)}$, which is itself bounded by $C \eta \|f\|_{L^\infty(-\eta,\eta)}$ according to Lemma \ref{lem:der_f_paw}. 
\end{proof}

We need a uniform bound in $\eta$ on the PAW eigenfunction $f$, in order to prove Theorem \ref{thm:E_PAW_bound}.

\begin{lem}
\label{lem:H1_boundedness_fPAW}
Let $f$ be an $L^2_\mathrm{per}$-normalized eigenfunctions associated to the first eigenvalue of \eqref{eq:PAW_eig_pb_pseudo}. Then there exists a positive constant $C$ independent of $\eta$ such that for all $0 < \eta \leq \eta_0$
$$
\| f \|_{H^1_\mathrm{per}} \leq C.
$$ 
\end{lem}

\begin{proof}
This is a direct consequence of Equation \eqref{eq:f_PAW_bounded}.
\end{proof}

We now have all the elements to complete the proof of Theorem \ref{thm:E_PAW_bound}.

\begin{proof}[Proof of the lower bound in Theorem \ref{thm:E_PAW_bound}]
Let $f$ be an $L^2_\mathrm{per}$-normalized generalized eigenfunction of the PAW eigenvalue problem \eqref{eq:PAW_eig_pb_pseudo}. By Proposition \ref{prop:H_PAW_formquad}, we have 
\begin{multline}
\label{eq:low_bnd_pawpseudo_central}
\psh{f}{H^{PAW}f} = \psh{f}{\widetilde{H}f} - 2 \sumlim{I \in \lbrace 0,a \rbrace}{} \psh{f - \psh{\widetilde{p}^I}{f}^T \widetilde{\Phi}_I}{ \psh{\widetilde{p}^I}{f}^T \left(H \Phi_I - (H + \delta V) \widetilde{\Phi}_I\right) } \\
+ \sumlim{I \in \lbrace 0,a \rbrace}{} \psh{f - \psh{\widetilde{p}^I}{f}^T \widetilde{\Phi}_I}{\delta V\left( f - \psh{\widetilde{p}^I}{f}^T \widetilde{\Phi}_I \right)}_{\eta,I}.
\end{multline}
We simply bound terms with $I=0$ as the terms with $I=a$ are treated exactly the same way.  First, we estimate $\psh{f - \psh{\widetilde{p}}{f}^T \widetilde{\Phi}}{\delta V\left( f - \psh{\widetilde{p}}{f}^T \widetilde{\Phi} \right)}_{I,\eta}$. By Lemma \ref{lem:f_paw_approx2}, we have:
\begin{align}
& \left| \psh{f - \psh{\widetilde{p}}{f}^T \widetilde{\Phi}}{\delta V\left( f - \psh{\widetilde{p}}{f}^T \widetilde{\Phi} \right)}_{\eta} \right| & \nonumber \\
& \qquad = Z_0 \left|   \left( f(0) - \psh{\widetilde{p}}{f}^T \widetilde{\Phi}(0) \right)^2 - \itg{-\eta}{\eta}{\chi_\eta (x) \left( f(x) - \psh{\widetilde{p}}{f}^T \widetilde{\Phi}(x) \right)^2  }{x} \right|  \nonumber \\
& \qquad = Z_0 \left|  \itg{-\eta}{\eta}{\chi_\eta (x)\left( \left( f(x) - \psh{\widetilde{p}}{f}^T \widetilde{\Phi}(x) \right)^2 - \left( f(0) - \psh{\widetilde{p}}{f}^T \widetilde{\Phi}(0) \right)^2 \right)}{x}  \right| \nonumber \\
& \qquad  \leq C \eta \left\| f' - \psh{\widetilde{p}}{f}^T \widetilde{\Phi}' \right\|_{\infty,\eta} \left\| f - \psh{\widetilde{p}}{f}^T \widetilde{\Phi} \right\|_{\infty,\eta} \nonumber \\
& \qquad \leq C \eta^2 \| f \|_{\infty,\eta}^2  ,
\label{eq:low_bnd_pawpseudo_1}
\end{align}
where in the last inequality, we applied \text{Lemma \ref{lem:f_paw_approx_DeltaV}}. 

We then estimate $\psh{f' - \psh{\widetilde{p}}{f}^T \widetilde{\Phi}'}{ \psh{\widetilde{p}}{f}^T (\Phi' - \widetilde{\Phi}') }$:
\begin{align}
& \left| \psh{f' - \psh{\widetilde{p}}{f}^T \widetilde{\Phi}'}{ \psh{\widetilde{p}}{f}^T (\Phi' - \widetilde{\Phi}') } \right| \nonumber \\
& \qquad = \left| \itg{-\eta}{\eta}{ \left( f_e'(x) - \psh{\widetilde{p}}{f}^T \widetilde{\Phi}'(x) \right) \psh{\widetilde{p}}{f}^T (\Phi' - \widetilde{\Phi}')(x)  }{x} \right| \nonumber \\
& \qquad \leq C \eta \left\| f_e' - \psh{\widetilde{p}}{f}^T \widetilde{\Phi}' \right\|_{\infty,\eta} \|f\|_{H^1_\mathrm{per}} \nonumber  \\
& \qquad \leq C \eta  \|f\|_{\infty,\eta}\|f\|_{H^1_\mathrm{per}}  ,
\label{eq:low_bnd_pawpseudo_2}
\end{align}
where in the first inequality, we used item 1 of Lemma \ref{lem:widetilde_p_stuff} and in the second, Lemma \ref{lem:f_paw_approx2}. Finally, it remains to estimate $\psh{f - \psh{\widetilde{p}}{f}^T \widetilde{\Phi}}{ \psh{\widetilde{p}}{f}^T (\delta_0 \Phi - \chi_\eta \widetilde{\Phi}) }$ :
\begin{align}
& \left| \psh{f - \psh{\widetilde{p}}{f}^T \widetilde{\Phi}}{ \psh{\widetilde{p}}{f}^T (\delta_0 \Phi - \chi_\eta \widetilde{\Phi}) } \right|  \nonumber \\
& \qquad \leq \left| \left(f(0) - \psh{\widetilde{p}}{f}^T \widetilde{\Phi}(0) \right) \psh{\widetilde{p}}{f}^T (\Phi(0) - \widetilde{\Phi}(0)) \right| \nonumber \\
& \qquad \qquad + \left| \itg{-\eta}{\eta}{\chi_\eta (x) \left( \left( f_e(x) - \psh{\widetilde{p}}{f}^T \widetilde{\Phi}(x) \right) \psh{\widetilde{p}}{f}^T \widetilde{\Phi}(x) - \left( f_e(0) - \psh{\widetilde{p}}{f}^T \widetilde{\Phi}(0) \right) \psh{\widetilde{p}}{f}^T \widetilde{\Phi}(0) \right)}{x} \right|  \nonumber\\
& \qquad \leq C \eta^2 \|f\|_{\infty,\eta}\|f\|_{H^1_\mathrm{per}} + C \eta \left\| \left( \left( f_e - \psh{\widetilde{p}}{f}^T \widetilde{\Phi} \right) \psh{\widetilde{p}}{f}^T \widetilde{\Phi} \right)' \right\|_{\infty,\eta} \nonumber.
\end{align}
We have
\begin{align*}
& \left\| \left( \left( f_e - \psh{\widetilde{p}}{f}^T \widetilde{\Phi} \right) \psh{\widetilde{p}}{f}^T \widetilde{\Phi} \right)' \right\|_{\infty,\eta} \\
& \qquad \leq \left\|  \left( f_e - \psh{\widetilde{p}}{f}^T \widetilde{\Phi} \right) \psh{\widetilde{p}}{f}^T \widetilde{\Phi}' \right\|_{\infty,\eta}  + \left\|  \left( f_e' - \psh{\widetilde{p}}{f}^T \widetilde{\Phi}' \right) \psh{\widetilde{p}}{f}^T \widetilde{\Phi} \right\|_{\infty,\eta} \\
& \qquad \leq  C \eta \|f\|_{\infty,\eta}\|f\|_{H^1_\mathrm{per}} + C \|f\|_{\infty,\eta}\|f\|_{H^1_\mathrm{per}},
\end{align*}
where we applied Lemma \ref{lem:f_paw_approx2} and items 3 and 4 of Lemma \ref{lem:widetilde_p_stuff}. Thus, 
\begin{equation}
\label{eq:low_bnd_pawpseudo_3}
\left| \psh{f - \psh{\widetilde{p}}{f}^T \widetilde{\Phi}}{ \psh{\widetilde{p}}{f}^T (\delta_0 \Phi - \chi_\eta \widetilde{\Phi}) } \right|  \leq C \eta  \|f\|_{\infty,\eta}\|f\|_{H^1_\mathrm{per}}.
\end{equation}

Injecting \eqref{eq:low_bnd_pawpseudo_1}, \eqref{eq:low_bnd_pawpseudo_2} and \eqref{eq:low_bnd_pawpseudo_3}, in \eqref{eq:low_bnd_pawpseudo_central}, we obtain
\begin{align*}
\psh{f}{H^{PAW}f} & \geq \psh{f}{\widetilde{H}f} - C\eta^2 \|f\|_{L^\infty}^2 -C\eta \|f\|_{\infty,\eta}\|f\|_{H^1_\mathrm{per}}\\
& \geq E_0 \psh{f}{\widetilde{S}f} - C\eta^2 \|f\|_{L^\infty}^2 -C\eta \|f\|_{\infty,\eta}\|f\|_{H^1_\mathrm{per}}.
\end{align*}
Using item 3 of Lemma \ref{lem:widetilde_S_L2_bound}, we obtain 
\begin{align*}
E_0 \psh{f}{S^{PAW} f} - C\eta^2 \|f\|_{L^\infty}^2 -C\eta \|f\|_{\infty,\eta}\|f\|_{H^1_\mathrm{per}} & \leq \psh{f}{H^{PAW}f} \\
& \leq E^{PAW} \psh{f}{S^{PAW}f},
\end{align*}
and the result follows from Lemma \ref{lem:H1_boundedness_fPAW} and the Sobolev embedding $\|f\|_{L^\infty} \leq C \|f\|_{H^1_\mathrm{per}}$. 
\end{proof}

\subsubsection{Improvement of the model}
\label{subsec:odd_PAW}

The critical term yielding the upper bound of Theorem \ref{thm:E_PAW_bound} is due to the poor approximation of $f$ by the pseudo wave functions $\widetilde{\phi}_k$. The latter are only even polynomials inside the cut-off region, hence incorporating odd functions to the PAW treatment should improve the upper bound on the PAW eigenvalue $E^{PAW}$.
\newline

The odd atomic wave functions are the functions 
\begin{equation}
\label{eq:tilde_theta}
\widetilde{\theta}_k(x) = \sin (2 \pi k x) , \quad k \in \N^*,
\end{equation}
which are eigenfunctions of the atomic Hamiltonian $-\frac{\mathrm{d}^2}{\mathrm{d}x^2} - Z_0 \sumlim{k \in \Z}{} \delta_k$. As these functions are already smooth, there is no need to take pseudo wave functions different from the atomic wave functions. 

To define the corresponding projector functions $\widetilde{q}_k$, first we denote by
\begin{equation}
\label{eq:G_odd_PAW}
G = \left( \itg{-\eta}{\eta}{\rho_\eta(t) \sin(2\pi j t) \sin (2\pi k t)}{t} \right)_{1 \leq j,k \leq N},
\end{equation}
where $\rho_\eta$ is the smooth cut-off function defined in Section \ref{subsec:VPAW1D}. $G$ is an invertible matrix since it is the Gram matrix of the linearly independent family of functions $(\sin(2\pi k x))_{1 \leq k \leq N}$. Now let $\widetilde{q}_k$ be defined by 
\begin{equation}
\label{eq:tilde_q}
\widetilde{q}_k(x) = \rho_\eta(x) \sumlim{j=1}{N} (G^{-1})_{jk} \widetilde\theta_j(x),
\end{equation}
so the functions $(\widetilde{\theta}_k)_{1 \leq k \leq N}$ and $(\widetilde{q}_k)_{1 \leq k \leq N}$ satisfy
$$
\psh{\widetilde{q}_j}{\widetilde{\theta}_k}= \delta_{jk}. 
$$
The functions $(\widetilde{\theta}_k^a)_{1 \leq k \leq N}$ are equal to $(\widetilde{\theta}_k(\cdot - a))_{1 \leq k \leq N}$ and the projector functions $(\widetilde{q}_k^a)_{1 \leq k \leq N}$ denotes the shifted projector functions $(\widetilde{q}_k(\cdot - a))_{1 \leq k \leq N}$.

Since $\widetilde{\theta}_k$ is an eigenfunction of $-\frac{\mathrm{d}^2}{\mathrm{d}x^2} - Z_0 \sumlim{k \in \Z}{} \delta_k$, for all $1\leq i,j \leq N$ and $I=0,a$,
$$
\psh{\widetilde{\theta}_i^I}{H \widetilde{\theta}_i^I} - \psh{\widetilde{\theta}_i^I}{H_\mathrm{ps} \widetilde{\theta}_i^I} = - \psh{\widetilde{\theta}_i^I}{-Z_I \chi_\eta \widetilde{\theta}_i^I}.
$$
Hence, the new expression of $H^{PAW}$ is given by
\begin{equation}
\label{eq:H_PAW_pseudo_odd}
\begin{split}
H^{PAW} = H_\mathrm{ps} & +  \sumlim{\substack{ i,j=1 \\ I \in \lbrace 0, a \rbrace}}{N} \widetilde{p}^I_i \left( \psh{\phi_i^I}{H \phi_j}_{I,\eta} - \psh{\widetilde{\phi}_i^I}{H_\mathrm{ps} \widetilde{\phi}^I_j}_{I,\eta} \right) \psh{\widetilde{p}^I_j}{\bm{\cdot}}  \\
&  -  \sumlim{\substack{ i,j=1 \\ I \in \lbrace 0, a \rbrace}}{N} \widetilde{q}_i^I \psh{\widetilde{\theta}_i^I}{-Z_I \chi_\eta \widetilde{\theta}_j}_{I,\eta}  \psh{\widetilde{q}_j^I}{\bm{\cdot}},
\end{split}
\end{equation}
and $S^{PAW}$ remains unchanged. 

We denote by $\widetilde{q}^I$ the vector of functions $(\widetilde{q}_1^I, \dots, \widetilde{q}_{N}^I)^T$ and $\widetilde{\Theta}_I$ the vector of functions $(\widetilde{\theta}_1^I, \dots, \widetilde{\theta}_{N}^I)^T$.
\newline

Using the functions $(\widetilde{\theta}_k^I)_{1 \leq k \leq N}$ and $(\widetilde{q}_k^I)_{1 \leq k \leq N}$ in the PAW treatment, we have the following theorem on the lowest PAW eigenvalue. 

\begin{theo}
\label{thm:E_PAW_improved_bound}
Let $\phi_i^I$, $\widetilde{\phi}_i^I$ and $\widetilde{p}_i^I$, for $i=1,\dots,N$ and $I=0,a$ be the functions defined in Section \ref{subsec:generation}. Suppose $\eta_0 > 0$ satisfies Assumption \ref{assump:1} and Assumption \ref{assump:2}. Let $(\widetilde{\theta}_k^I)_{1 \leq k \leq N}$ be the functions given by  \eqref{eq:tilde_theta} and $(\widetilde{q}_k^I)_{1 \leq k \leq N}$ be the functions given by \eqref{eq:tilde_q}.
Let $E^{PAW}$ the lowest eigenvalue of the generalized eigenvalue problem  $H^{PAW} f = E^{PAW} S^{PAW} f$ with $H^{PAW}$ defined in \eqref{eq:H_PAW_pseudo_odd}. Let $E_0$ be the lowest eigenvalue of $H$ \eqref{eq:H_mol}. Then there exists a positive constant $C$ independent of $\eta$ such that for all $0 < \eta \leq \eta_0$
\begin{equation}
- C \eta \leq E^{PAW} - E_0 \leq  C \eta^{2N}.
\end{equation}
\end{theo}

The proof of Theorem \ref{thm:E_PAW_improved_bound} follows the same steps of the proof of Theorem \ref{thm:E_PAW_bound}. First, we prove that for $g \in H^1_\mathrm{per}$, the quantity $\psh{g}{H^{PAW}g}$ is equal to $\psh{g}{\widetilde{H}g}$ and error terms of the form  $g - \psh{\widetilde{p}^I}{g}^T \widetilde{\Phi}_I  - \psh{\widetilde{q}^I}{g}^T \widetilde{\Theta}_I$ that needs to be estimated.

\begin{prop}
\label{prop:H_PAW_formquad2}
Let $g \in H^1_{\mathrm{per}}(0,1)$. Then
\begin{equation*}
\begin{split}
\psh{g}{H^{PAW}g} = \, & \psh{g}{\widetilde{H}g} - 2 \sumlim{I \in \lbrace 0, a \rbrace}{} \psh{g - \psh{\widetilde{p}^I}{g}^T \widetilde{\Phi}_I}{ \psh{\widetilde{p}^I}{g}^T \left(H \Phi_I - (H + \delta V) \widetilde{\Phi}_I\right) } \\
& + 2 \sumlim{I \in \lbrace 0, a \rbrace}{}  \psh{g - \psh{\widetilde{q}^I}{g}^T \widetilde{\Theta}_I}{ \psh{\widetilde{q}^I}{g}^T  \delta V\widetilde{\Theta}_I }_{I,\eta} \\
& + \sumlim{I \in \lbrace 0, a \rbrace}{}  \psh{g - \psh{\widetilde{p}^I}{g}^T \widetilde{\Phi}_I - \psh{\widetilde{q}^I}{g}^T \widetilde{\Theta}_I}{\delta V\left( g - \psh{\widetilde{p}^I}{g}^T \widetilde{\Phi}_I  - \psh{\widetilde{q}^I}{g}^T \widetilde{\Theta}_I \right)}_{I,\eta} .
\end{split}
\end{equation*}
\end{prop}

\begin{proof}
The proof is similar to the proof of Proposition \ref{prop:H_PAW_formquad}. 
\end{proof}

\begin{lem}
\label{lem:T_odd_L2}
There exists a constant $C$ independent of $\eta$ such that for all $f \in H^1_\mathrm{per}(0,1)$ for $x$ in $(-\eta, \eta)$,
\begin{equation*}
\left| \psh{\widetilde{q}^I}{f}^T \widetilde{\Theta}_I(x) \right| \leq C \|f\|_{\infty, \eta, I}.
\end{equation*}
\end{lem}

\begin{proof} For clarity, we will drop the index $I$. For $0 \leq j \leq N-1$, let
$$
v_j = (2 \pi \eta)^{2j+1} \begin{pmatrix}
1 \\
2^{2j+1} \\
\vdots \\
N^{2j+1}
\end{pmatrix}, \qquad \hat{v}_j = \frac{1}{\eta^{2j+1}} v_j.
$$
Let $(\hat{w}_j)_{0 \leq j \leq N-1}$ be the dual basis of $(\hat{v}_j)_{0 \leq j \leq N-1}$ and $w_j = \frac{1}{\eta^{2j+1}} \hat{w}_j$. Let $M$ be the matrix such that for all $0 \leq j,k\leq N-1$,
$$
M_{jk} =\frac{(-1)^{j+k}}{(2j+1)!(2k+1)!} \itg{-1}{1}{\rho(t)t^{2j+2k+2}}{t}.
$$
By a Taylor expansion, we obtain for $t \in (-1,1)$,
$$
\widetilde{\Theta}(\eta t) = \begin{pmatrix}
\sin (2\pi \eta t) \\
\vdots \\
\sin (2 \pi \eta Nt)
\end{pmatrix} = \sumlim{k=0}{N-1} \frac{(-1)^k(2\pi \eta t)^{2k+1}}{(2k+1)!} \begin{pmatrix}
1 \\
2^{2j+1} \\
\vdots \\
N^{2j+1}
\end{pmatrix} + R_{\widetilde{\Theta}}(\eta t),
$$
where $|R_{\widetilde{\Theta}}(\eta t)| \leq C \eta^{2N+1}$. Then, we can rewrite the matrix $G$ given by \eqref{eq:G_odd_PAW} 
\begin{align*}
G & = \eta \itg{-1}{1}{\rho(t) \left( \sumlim{j=0}{N-1} \frac{(-1)^j \eta^{2j+1}}{(2j+1)!} \hat{v}_j t^{2j+1} + R_{\widetilde{\Theta}}(\eta t) \right) \left( \sumlim{k=0}{N-1} \frac{(-1)^k \eta^{2k+1}}{(2k+1)!} \hat{v}_k t^{2k+1} + R_{\widetilde{\Theta}}(\eta t) \right)^T }{t}  \\
& = \eta \sumlim{j,k=0}{N-1} M_{jk} v_j v_k^T + \eta \itg{-1}{1}{\sumlim{j=0}{N-1} \frac{(-1)^j }{(2j+1)!} t^{2j+1} \left(v_jR_{\widetilde{\Theta}}(\eta t)^T + R_{\widetilde{\Theta}}(\eta t)v_j^T \right) }{t} + \cO(\eta^{4N+3}).
\end{align*}
Hence, we have for $0 \leq j,k \leq N-1$,
$$
w_j^T G w_k = \eta M_{jk} + \eta w_j^T \mathcal{R}_k + \eta \mathcal{R}_j^T w_k + \cO (\eta^5), 
$$
where 
$$
\mathcal{R}_k = \itg{-1}{1}{\rho(t) \frac{(-1)^k }{(2k+1)!} t^{2k+1} R_{\widetilde{\Theta}}(\eta t)}{t} .
$$
But $\|w_k\|= \cO (\eta^{-2k-1})$ and $|R_{\widetilde{\Theta}}(\eta t)| \leq C \eta^{2N+1}$, hence $\mathcal{R}_j^T w_k = \cO(\eta^2)$. Thus, if we denote by 
$$
W = \begin{pmatrix}
w_0^T \\
\vdots \\
w_{N-1}^T
\end{pmatrix}, \qquad V = \begin{pmatrix}
v_0^T \\
\vdots \\
v_{N-1}^T
\end{pmatrix},
$$
we obtain
$$
W G W^T = \eta M + \cO(\eta^3),
$$
and 
\begin{equation}
\label{eq:G_-1}
W^{-T} G^{-1} W^{-1} = V G^{-1} V^T = \frac{1}{\eta} M^{-1} + \cO \left(\eta\right).
\end{equation}
Thus, we have for $f \in L^\infty(-\eta,\eta)$ and $x \in (-\eta,\eta)$
\begin{multline}
\label{eq:qf_Theta}
\psh{\widetilde{q}}{f}^T \widetilde{\Theta}(x)  = \eta \left(\itg{-1}{1}{\rho(t) f(\eta t) G^{-1} \left( \sumlim{j=0}{N-1} \frac{(-1)^j}{(2j+1)!} t^{2j+1} v_j + R_{\widetilde{\Theta}}(\eta t)^T \right)}{t} \right)^T \\
\left( \sumlim{j=0}{N-1} \frac{(-1)^j}{(2j+1)!} \left( \frac{x}{\eta} \right)^{2j+1} v_j + R_{\widetilde{\Theta}}(x) \right).
\end{multline}
By expanding \eqref{eq:qf_Theta}, three types of terms arise involving
\begin{enumerate}
\item $v_j^T G^{-1} v_k$: by \eqref{eq:G_-1}, we have $| v_j^T G^{-1} v_k| = \cO \left(\tfrac{1}{\eta}\right)$;
\item $v_j^T G^{-1} R_{\widetilde{\Theta}}(x)$: by \eqref{eq:G_-1}, $\|v_jG^{-1}\| = \cO\left(\tfrac{1}{\eta^{2N-1}}\right)$ and because $R_{\widetilde{\Theta}}(x) = \cO(\eta^{2N+1})$, we have $|v_j^T G^{-1} R_{\widetilde{\Theta}}(x)| = \cO \left(\eta^2\right)$;
\item $R_{\widetilde{\Theta}}(\eta t)^T G^{-1} R_{\widetilde{\Theta}}(x)$: by \eqref{eq:G_-1}, we deduce that $\|G^{-1}\| = \cO\left(\tfrac{1}{\eta^{4N-1}}\right)$, but $R_{\widetilde{\Theta}}(x) = \cO(\eta^{2N+1})$, hence $|R_{\widetilde{\Theta}}(\eta t)^T G^{-1} R_{\widetilde{\Theta}}(x)| = \cO\left(\eta^3\right)$.
\end{enumerate}
Thus, 
$$
|\psh{\widetilde{q}}{f}^T \widetilde{\Theta}(x)| \leq C \|f\|_{\infty,\eta}.
$$
\end{proof}

\begin{lem}
\label{lem:f_odd_approx}
Let $f$ be a smooth and odd function. Then we have 
\begin{equation*}
\left\| f - \psh{\widetilde{q}^I}{f}^T \widetilde{\Theta}_I \right\|_{\infty, \eta, I} \leq C \eta^{2N+3} \|f^{(2N+3)}\|_{\infty, \eta, I}, 
\end{equation*}
\end{lem}

\begin{proof}
We simply write the Taylor expansion of $f$ around 0. Then by expanding the functions $\theta_k$ around $0$, it is easy to show that the difference between $f$ and the best approximation in $(-\eta,\eta)$ of $f$ by a linear combination of $\theta_k$ is bounded by the Taylor remainder of $f$ and terms arising from the truncation of the expansions of the functions $\theta_k$ which are both of order $\cO(\eta^{2N+3})$. We then conclude using Lemma \ref{lem:T_odd_L2}.
\end{proof}

The presence of $\widetilde{\theta}_j$ and $\widetilde{q}_j$ (see \eqref{eq:H_PAW_pseudo_odd} above) does not change the lower bound of the PAW eigenvalue as it does not improve the estimate of critical terms in the proof of lower bound in Theorem \ref{thm:E_PAW_bound}. However, we get a much better upper bound as it is the odd part of the wave function $\psi$ which prevents to have a better bound. Thus introducing these odd functions in the PAW treatment, we have Theorem \ref{thm:E_PAW_improved_bound}. 

\section{Numerical tests}
\label{sec:numerics}

In this section, some numerical tests are provided to confirm the bounds obtained in Theorems \ref{thm:bound_trunc_PAW}, \ref{thm:E_PAW_bound} and \ref{thm:E_PAW_improved_bound}. The simulations of the different PAW versions are done with $a=0.4$ and $Z_0 = Z_a =10$. 

\subsection{The PAW equations}
\label{subsec:PAW}

\subsubsection{Without pseudopotentials}
\label{subsec:PAW_simu_sans_pseudo}

We solve the generalized eigenvalue problem 
\begin{equation*}
H^N f = E^{(\eta)} S^N f ,
\end{equation*}
where $H^N$ and $S^N$ are defined by Equations (\ref{eq:H_N}) and (\ref{eq:S_N}), by expanding $f$ in 512 plane-waves. We study how $E^{(\eta)}$ behaves as a function of $\eta$. In our case, the PAW eigenvalue $E^{(\eta)}$ is smaller than $E_0$. For this regime, Theorem \ref{thm:bound_trunc_PAW} states that $E^{(\eta)}$ converges at least linearly to $E_0$, which is what we observe in Figure \ref{grph:PAW_tronque}.

\begin{figure}[H]
\centering
\includegraphics[width = 0.7\textwidth]{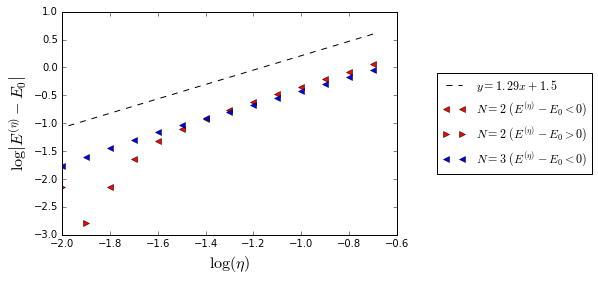}
\caption{Error on the lowest eigenvalue of the truncated PAW equations \eqref{eq:PAW_eig_pb}}
\label{grph:PAW_tronque}
\end{figure}

\subsubsection{With pseudopotentials}

The eigenfunction $f$ is expanded in $1000$ plane waves for which convergence is reached. The integrals of plane-waves against PAW functions are computed with an accurate numerical integral scheme. 

In view of Figure \ref{grph:PAW_pseudo}, the lower bound in Theorem \ref{thm:E_PAW_bound} seems sharp. The use of odd PAW functions improves the error on the PAW eigenvalue (Figure \ref{grph:PAW_pseudo_odd}) for a range of moderate values of the cut-off radius $\eta$. However, the use of odd PAW functions does not give a better lower bound. 

Finally, the upper bound in Theorem \ref{thm:E_PAW_improved_bound} seems optimal (see Figure \ref{grph:PAW_pseudo_odd}). For $N=2$, we have a slope close to the theoretical value ($2N=4$).  

\begin{figure}[H]
\centering
\includegraphics[width = 0.7\textwidth]{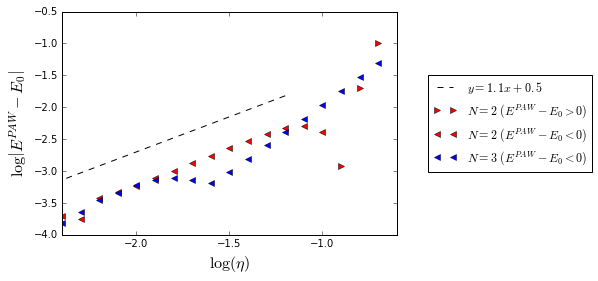}
	
\caption{Error on the lowest eigenvalue of the PAW equations \eqref{eq:PAW_eig_pb_pseudo} with pseudopotentials}
\label{grph:PAW_pseudo}
\end{figure}

\begin{figure}[H]
\centering

\includegraphics[width = 0.7\textwidth]{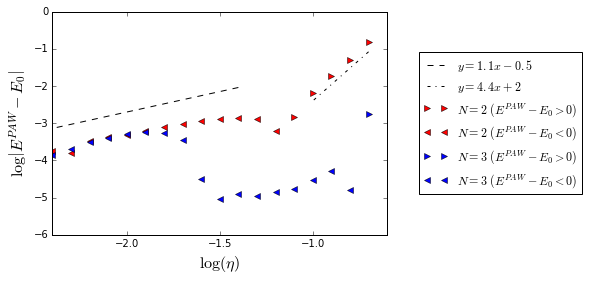}
	
\caption{Error on the lowest eigenvalue of the PAW equations with pseudopotentials including odd PAW functions}
\label{grph:PAW_pseudo_odd}
\end{figure}

\subsection{Comparison between the PAW and VPAW methods in pre-asymptotic regime}

The simulations are run for a fixed value of $d =6$ and different values of $\eta = 0.1$ and $\eta = 0.2$. In Figure \ref{grph:PAW-VPAW}, $E_0$ is the lowest eigenvalue of the 1D-Schr\"odinger operator $H$. The PAW method considered in Figure \ref{grph:PAW-VPAW} is the generalized eigenvalue problem \eqref{eq:PAW_eig_pb_pseudo}.

Using Fourier methods to solve the VPAW eigenvalue problem \eqref{eq:H_VPAW_eig_pb}, we have the following bound on the computed eigenvalue $E_M^{\scriptscriptstyle \mathrm{VPAW}}$ \cite{blanc2017vpaw1d}:
\begin{equation}
\label{eq:rappel}
0 < E_M^{\scriptscriptstyle \mathrm{VPAW}} - E_0 \leq C \left( \frac{\eta^{4N}}{M} + \frac{1}{\eta^{2d-2}} \frac{1}{M^{2d-1}} \right),
\end{equation}
where $M$ is the number of plane-waves, $N$ the number of PAW functions and $d$ the regularity of the PAW pseudo wave functions $\widetilde{\phi}_k$.

As expected, the PAW method quickly converges to $E^{PAW}$ which, according to Theorem \ref{thm:E_PAW_bound}, is close but not equal to $E_0$. Although the VPAW method does not remove the Dirac singularities -which is why, asymptotically, the VPAW method convergence rate is of order $\mathcal{O}\left( \frac{1}{M} \right)$-, it converges faster to $E_0$ than the PAW method with pseudopotentials. 



\begin{figure}[H]
\centering
	\begin{subfigure}[b]{0.9\textwidth}
	\includegraphics[width=\textwidth]{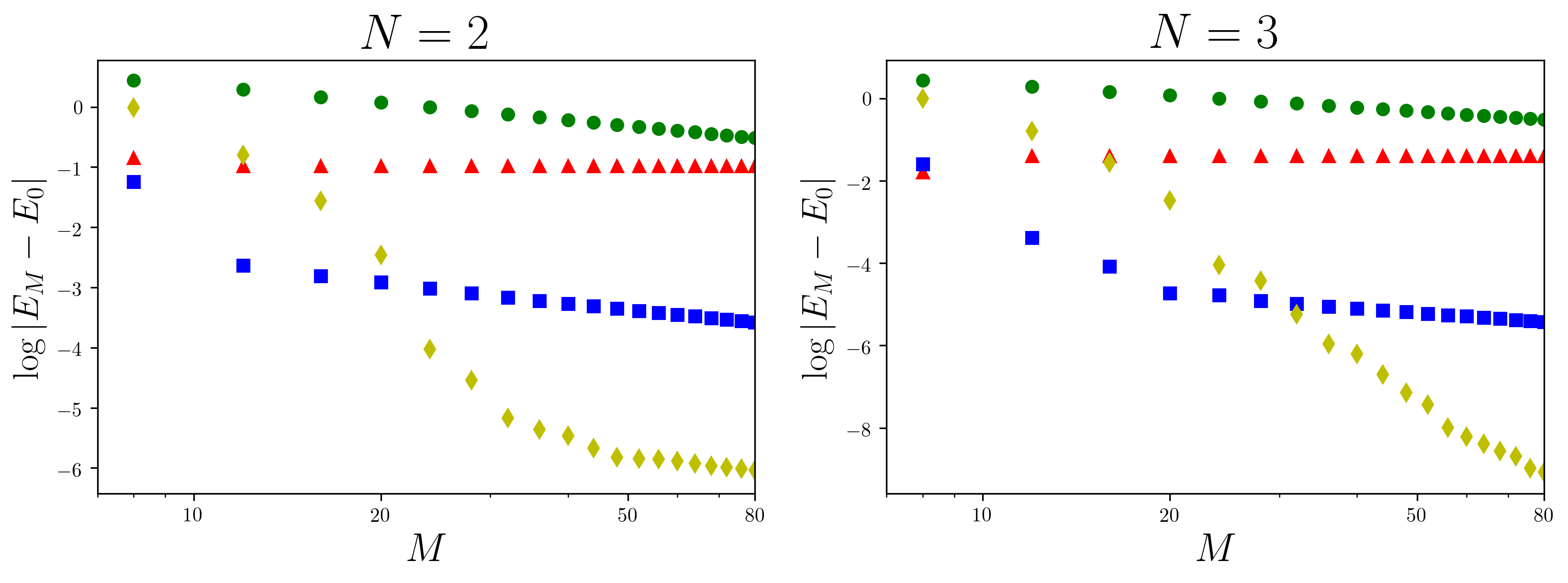}
	\end{subfigure}
	
	\begin{subfigure}[b]{0.7\textwidth}
	\includegraphics[width=\textwidth]{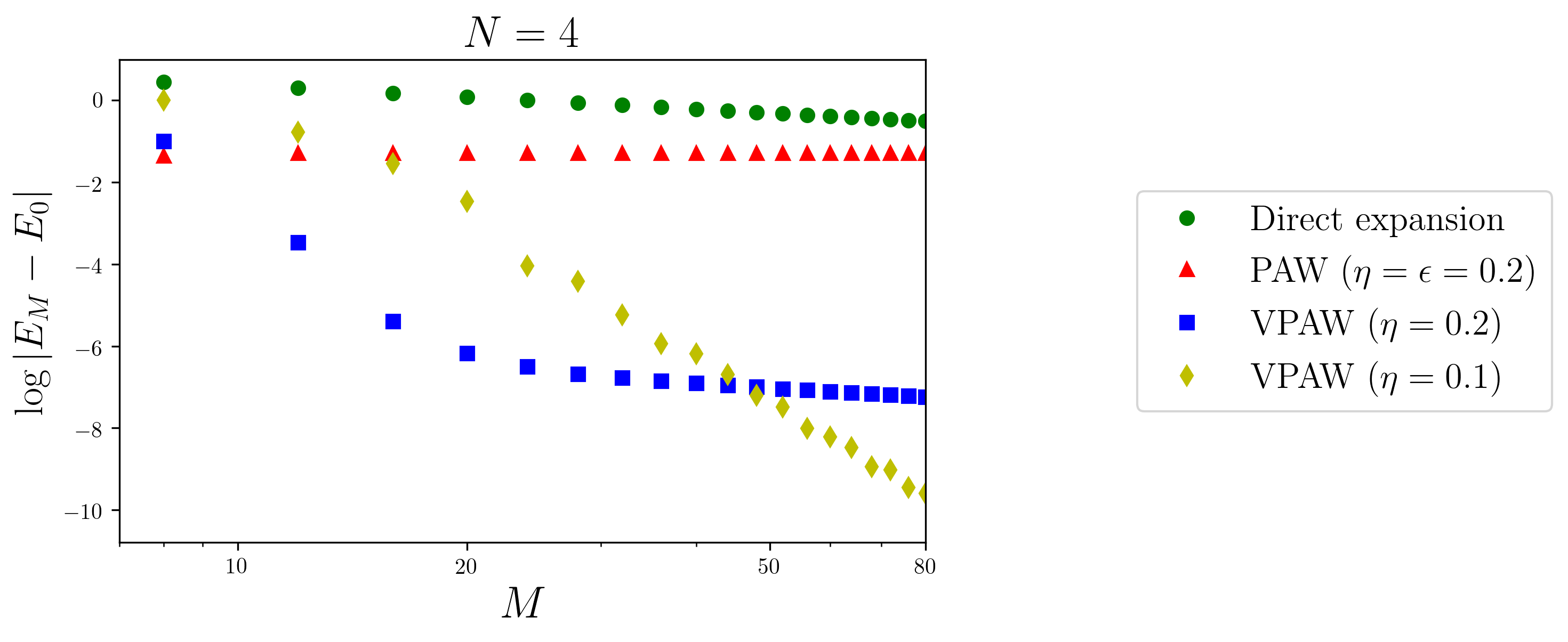}
	\end{subfigure}
\caption{Comparison between the PAW and VPAW methods}
\label{grph:PAW-VPAW}
\end{figure}


\end{document}